\documentclass{article}

\usepackage{arxiv}

\usepackage[utf8]{inputenc} % allow utf-8 input
\usepackage[T1]{fontenc}    % use 8-bit T1 fonts
\usepackage{url}            % simple URL typesetting
\usepackage{booktabs}       % professional-quality tables
\usepackage{amsfonts}       % blackboard math symbols
\usepackage{nicefrac}       % compact symbols for 1/2, etc.
\usepackage{microtype}      % microtypography
\usepackage{lipsum}

\usepackage{relsize,balance,lipsum,bbm,enumerate,times,comment,color,graphicx,setspace,mathdots,mathrsfs,amssymb,latexsym,amsfonts,amsmath,cite,stmaryrd,caption,pgf,accents,mathtools,tabu,enumitem,hhline,array,epstopdf,nicefrac,amsthm,microtype,algorithmic,array,float,bm,url}
\newtheorem{theorem}{Theorem}
\newtheorem{definition}{Definition}

\newtheorem{lemma}{Lemma}
\newtheorem*{remark}{Remark}

\newtheorem{assumption}{Assumption}
\usepackage{graphicx}
\usepackage{mathtools} 

\usepackage[utf8]{inputenc}
\usepackage[english]{babel}

\newcommand{\sa}[1]{{\textcolor{black}{#1}}}

\DeclareMathOperator*{\argmin}{arg\,min}
\usepackage{tikz}

\usetikzlibrary{positioning}
\definecolor{mygreen}{RGB}{153,255,153}
\definecolor{myorange}{RGB}{255,178,102}
\definecolor{myred}{RGB}{255,153,153}
\definecolor{myblue}{RGB}{153,204,255}

\title{Existence of Solutions for Non-monotone Variational Inequalities and Implications for Games
}

\author{
    Sina Arefizadeh$^*$\\
    Dept. of Electrical and Computer Engineering\\
    Arizona State University\\
    Tempe, Arizona, USA\\
    \texttt{sarefiza@asu.edu}
\And
    Angelia~Nedi\'c\\
    Dept. of Electrical and Computer Engineering\\
    Arizona State University\\
    Tempe, Arizona, USA\\
    \texttt{Angelia.Nedich@asu.edu}
}

\theoremstyle{definition}
\newtheorem{example}{Example}
\def\argmin{\mathop {\rm argmin}}

\begin{document}
\thanks{Corresponding author}
\thanks{ This work has been supported in parts by the NSF awards CCF-2106336 and CIF-2134256.}
\thanks{This paper has been published in the Journal of Optimization Theory and Applications. Please cite the published version: Arefizadeh, S., Nedić, A. Existence of Solutions for Non-monotone Variational Inequalities and Implications for Games. J Optim Theory Appl 210, 23 (2026). https://doi.org/10.1007/s10957-026-03055-6
}
\maketitle
\begin{abstract}
In this paper, we study the existence of solutions in non-monotone variational inequalities (VIs) through the normal mapping properties. In particular, {\color{black} we provide sufficient conditions for the existence of solutions assuming that the normal mapping associated with the VI is norm coercive and its generalized Jacobian has certain properties, such as a full rank at points where the normal mapping is not zero. Then, we investigate sufficient conditions for the VI mapping and its Jacobian that ensure the generalized Jacobian of the normal mapping has full rank}, such as
the uniform P-function and the uniform P-matrix condition.
Subsequently, {\color{black} we focus on VIs arising from games and interpret our results in a game setting and provide} a sufficient condition for a game to have a Nash equilibrium.
Through examples, we show that our sufficient conditions can be used to assert the existence of a solution to a VI, or a quasi-Nash in a game, while the existing results relying on the uniform P-function property or the P$_\Upsilon$-matrix condition cannot be employed. 
\end{abstract}

%%%%%%%%%%%%%%%%%%%%%%%%%
\section{Introduction}
%%%%%%%%%%%%%%%%%%%%%%%%%
In this paper, we study variational inequalities (VIs) with the goal of establishing new sufficient conditions that guarantee the existence of solutions for non-monotone VIs, as well as the existence of equilibria in the special case when the VI is induced by a competitive game among a set of players. 
The conditions ensuring the existence of solutions have been studied for both monotone and non-monotone VIs. In particular, certain conditions on the mapping,
$F$ such as continuity, pseudo-monotonicity, or structural properties of its Jacobian, have been employed to assert the existence of solutions \cite{facchinei2003finite,parise2019variational,ravat2011characterization}.
%THIS PAPER HAS NOTHING TO DO WITH ALGORITHMS
%Given the existence of solutions, there are classic methods~\cite{facchinei2003finite,konov} and many recent papers  developing solution algorithms, assuming conditions on the type of smoothness or continuity of mappings \cite{hsieh2020explore,vankov2024generalized,choudhury2023single,pmlr-v283-xiao25a,choudhury2025extragradient,chen2023generalized,vankov2023last}, to name a few.

In the setting of games,
a solution concept in classical game theory is the Nash equilibrium. \sa{The existence of a Nash equilibrium can be guaranteed in games with convex cost functions and compact convex action sets. This guarantee may fail in the absence of convexity~\cite{pang2011nonconvex}.}
The existence of a Nash equilibrium has been extensively analyzed through the variational inequality framework, as discussed in \cite{naghizadeh2017uniqueness,ebrahimi2025united,melo2018variational,parise2019variational}, which rely on the strict monotonicity of the game mapping. Recent works investigating the algorithms to obtain the Nash equilibrium, such as \cite{li2023stackelberg,carnevale2024tracking}, focus on games with strongly monotone mappings, while others (e.g., \cite{arefizadeh2022distributed,qu2025generalized,ran2024distributed}) study games with compact convex action sets. 
% \an{\bf Sina, these references are on algorithms for finding a solution - please, cite papers that study the existence of solutions -- see the end of the sentence --->}
Papers~\cite{gohary2009generalized,hobbs2007nash,luo1996mathematical,scutari2010convex} argue that the assumptions such as bounded action sets for players or monotone mappings are often unrealistic. Consequently, when games lack convexity of the players' cost functions, the Nash equilibrium no longer serves as a suitable solution concept.

In games where players exhibit bounded rationality rather than complete rationality, the players may lack the ability to perform an exhaustive search of their action sets, which is required to act in accordance with the maximum utility theorem \cite{aleskerov2007utility}. Under the bounded rationality assumption, players instead tend to settle for locally superior alternatives, rather than globally optimal ones. Such alternatives are often chosen based on incremental improvements or inclinations toward relatively better options \cite{simon1955behavioral}. In such scenarios, papers~\cite{pang2011nonconvex,pang2013joint} suggest studying a local Nash equilibrium and introduce the concept of {a quasi-Nash equilibrium, as a more general solution concept in games with non-convex players' cost functions. 
While a quasi-Nash equilibrium arises in non-convex games, higher-order optimality conditions are typically required to distinguish between a quasi-Nash equilibrium and a local Nash equilibrium~\cite{pang2011nonconvex}.
%Studying quasi-NEs \an{appears naturally to be} the first step towards evaluating the existence of an local-NE, since once the existence of a quasi-NE is ensured, a second-order test can be applied to verify the existence of an local-NE \cite{pang2011nonconvex}. 
The study~\cite{ratliff2013characterization} characterizes the local Nash equilibrium in games where each player cost function is locally convex with respect to the player's own decision variable. 
%The authors also provide a steepest descent algorithm with a provable sufficient condition for local convergence to the local-NEs. }
% \an{\bf Sina, please see their abstract: We present derivative–based necessary and sufficient conditions ensuring player strategies constitute local Nash
% equilibria in non–cooperative continuous games. Our results
% can be interpreted as generalizations of analogous second–order
% conditions for local optimality from nonlinear programming
% and optimal control theory. Drawing on this analogy, we
% propose an iterative steepest descent algorithm for numerical
% approximation of local Nash equilibria and provide a sufficient
% condition ensuring local convergence of the algorithm. We
% demonstrate our analytical and computational techniques by
% computing local Nash equilibria in games played on a finite–
% dimensional differentiable manifold or an infinite–dimensional
% Hilbert space. NO Convexity other than perhaps in the constraint set.} \sa{Please see the deffinition~3 of their paper. local strictly convexity is an intrinsic element of their definition.}

This paper builds on some insights from our prior work reported in a conference paper \cite{NonMonotoneVI} and its extended version on arxiv
\cite{2510.02724}, where sufficient conditions for the existence of solutions to an unconstrained VI$(\mathbb{R}^m,F)$ was investigated through the properties of the Jacobian $\nabla F(\cdot)$ and the natural mapping $F^{\rm nat}_K(\cdot)$.  
Unlike our prior work in \cite{NonMonotoneVI}, 
\cite{2510.02724}, in this paper, we explore sufficient conditions for the existence of solutions to generic non-monotone VI$(K,F)$ problem and study the implications for the VIs arising from games. Aside from going beyond
the unconstrained VIs, we relax the P-function 
property of the VI mapping and the related P-matrix condition, which have often been used to ensure the existence of solutions for non-monotone VIs~\cite{facchinei2003finite,parise2019variational}.

{\color{black}
In this paper, we provide sufficient conditions for the existence of solutions to a non-monotone VI$(K,F)$. In the development, we 
explore the implications of several P-properties for the singular values of the Jacobian $\nabla F(\cdot)$, which are used to
obtain sufficient conditions for the existence of solutions to  VI$(K,F)$. These conditions involve the norm coercivity property of the normal mapping and some assumptions on the generalized Jacobian of the normal mapping. Then, we explore the implications of several P-properties for the singular values of the Jacobian $\nabla F(\cdot)$.}
 Applying results to VIs arising from games, we obtain new results for the existence of a Nash equilibrium that extend those developed in~\cite{parise2019variational}. 
Specifically,
the novel results of this paper include:
\begin{enumerate}
\item In Section~\ref{Sec-VI-Problem Statement}, 
for a mapping $x\mapsto A(x)\in\mathbb{R}^{m\times m}$ that satisfies the uniform P-matrix condition on a set $K$, we show that all the singular values of all principal sub-matrices of $A(x)$ are uniformly bounded away from zero on the set $K$ (Theorem~\ref{Thm-Singval-Pmat}). We also establish that, when a mapping $F:K\to\mathbb{R}^m$ is continuously differentiable and has a uniform P-function property, then all principal sub-matrices of the Jacobian $\nabla F(x)$ have singular values bounded away from zero uniformly over $K$, for a Cartesian and non-Cartesian structure of the set $K$ (Theorem~\ref{thm-P-function}).

\item In Section~\ref{Sec-Main result-VI}, {\color{black} we establish our main results on the existence of a solution to VI$(K,F)$ assuming the norm coercivity of the normal mapping $F_K^{\rm nor}(\cdot)$ and some properties of its generalized Jacobian (Theorem~\ref{thm-Normal-Mapping_Sol_Existence0} and Theorem~\ref{thm-Normal-Mapping_Sol_Existence}). Subsequently, we study properties of the mapping $F(\cdot)$ and the set $K$ ensuring that the sufficient conditions are satisfied} (Lemma~\ref{Lem-coercivity of F-K-Nor}, Theorem~\ref{thm-general jac of Const}). We also provide an example showing that our result can be used to assert the existence of a solution to a VI, while the existing result relying on the uniform P-function property cannot be used.
\item In Section~\ref{Sec-Game}, we consider a VI$(K,F)$ arising from a game and investigate the implications of the uniform P-matrix condition for principal minors of the Jacobian $\nabla F(\cdot)$ and norm coercivity of $F(\cdot)$ (Theorem~\ref{thm:Pmatrix_strongly_convex}, Theorem~\ref{thm:Pmatrix_norm_coercive}). 
We also investigate the implications of the $\rm P_{\Upsilon}$-matrix condition for the game Jacobian (Theorem~\ref{thm-P_Upsilon-uniquesol}). We show that these conditions imply that the conditions of our main existence result are satisfied.
A particularly interesting finding is that: the uniform P-matrix and $\rm P_{\Upsilon}$-matrix conditions imply that each player's cost function is strongly convex in the player's own decision variable (Theorem~\ref{thm:Pmatrix_strongly_convex}(a), Theorem~\ref{thm-P_Upsilon-uniquesol}(a)); hence, under such conditions every solution to the VI formulation of the game is in fact a Nash equilibrium for the game. 
Moreover, if a game has a quasi-Nash equilibrium $\hat x$ and the players' cost functions satisfy a Polyak–\L{}ojasiewicz (P\L{})-type condition, then $\hat x$ is a Nash equilibrium (Theorem~\ref{Thm-Our+PL}).
We also give an example showing that our main result can be used to assert the existence of a Nash equilibrium in a game, while 
the results relying on the uniform P-matrix or the $\rm P_{\Upsilon}$-matrix condition are not applicable.
\end{enumerate}

The remainder of the paper is organized as follows: 
Section~\ref{Sec-VI-Problem Statement} states the VI problem, and introduces some concepts and results related to norm coercivity of the mapping and non-singularity of its Jacobian, which we use in the subsequent development. 
{\color{black} Section~\ref{Sec-Main result-VI} provides sufficient conditions for the existence of solutions to
VI$(K,F)$, and explores the structure of the mapping $F(\cdot)$ and the set 
$K$ for the conditions to hold.}
 Section~\ref{Sec-Game} considers the implications of the main result of Section~\ref{Sec-Main result-VI} for VIs arising from games. It also establishes a new result for the existence of a Nash equilibrium.
Section~\ref{Sec-Conclusion} concludes the paper.

%---------------------------------------------------
\section{Preliminaries} \label{Sec-VI-Problem Statement}
%---------------------------------------------------
In this section, we introduce the VI problem~\cite{facchinei2003finite,konov,KSbook}, and provide some basic notions and relation that we will use later on to establish sufficient conditions for the existence of solutions to VIs and games.

\begin{definition}
%[VI Problem~\cite{facchinei2003finite}] 
Given a set $K\subseteq\mathbb{R}^m$ and a mapping $F:K\to\mathbb{R}^m$, the variational inequality problem, denoted by
VI$(K,F)$, consists of determining
   a point $x^*\in K$ such that 
   $$\langle F(x^*),x-x^*\rangle\geq 0\qquad\hbox{for all $x\in K$}.$$ 
   \end{definition}
A point $x^*$ satisfying the preceding inequality is referred to as a solution to the variational inequality problem VI$(K,F)$. The set of all such solutions is denoted by SOL$(K,F)$. 

A VI$(K,F)$ is said to be monotone if the mapping $F(\cdot)$ is monotone on the set $K$, i.e.,
\[\langle F(x)-F(y),x-y\rangle\ge0\qquad\hbox{for all }x,y\in K.\]
A VI$(K,F)$ is said to be non-monotone, when $F(\cdot)$ fails to be monotone on $K$. Our interest is in non-monotone VIs, where the mapping $F(\cdot)$  is differentiable.

We will say that the mapping $F:K\to\mathbb{R}^{m}$ is norm coercive on the set $K$ when the following relation is satisfied
\begin{equation*}
\lim_{\|x\|_2\to\infty,x\in K} \|F(x)\|_2=+\infty.\end{equation*}
If $K=\mathbb{R}^m$ in the preceding relation, we say that $F(\cdot)$ is norm coercive.

 Next, we provide the existence of a solution result to an unconstrained VI$(\mathbb{R}^m,F)$. We let $\nabla F(\cdot)$ and $\det\left(\nabla F(\cdot)\right)$ denote, respectively, the Jacobian of the mapping $F(\cdot)$ and its determinant.
 \begin{theorem}[Theorem 6 \cite{NonMonotoneVI}] \label{Thm-main theorem-prime}
    Let mapping $F:\mathbb{R}^{m}\to\mathbb{R}^{m}$ be differentiable and norm coercive, and assume that 
    $\det\left(\nabla F(x)\right)\neq 0$
    for every $x\in \mathbb{R}^{m}$ where $F(x)\neq 0$. Then,  the {\rm VI}$(\mathbb{R}^{m},F)$ has a solution.
\end{theorem}

{\color{black} Throughout the paper, when stated that the mapping $F:K\to\mathbb{R}^m$ is locally Lipschitz continuous, or continuously differentiable, on the set $K\subseteq \mathbb{R}^m$ it means that $F(\cdot)$ has these properties on an open set containing the set $K$.

In the next two sections, we consider uniform P-function and 
uniform P-matrix conditions for the mapping $F(\cdot)$ in relation to 
the norm coercivity of the mapping $F(\cdot)$ and non-singularity of the Jacobian $\nabla F(\cdot)$, respectively.
}

%------------------------------------------------------
\subsection{Uniform P-function}\label{subs-unip-fun}
%------------------------------------------------------

{\color{black} The uniform P-matrix condition, to be defined shortly, can be interpreted as a relaxation of the strong monotonicity property requiring that: for some $\mu>0$,
\[\langle F(x)-F(y),x-y\rangle \ge \mu\|x-y\|_2^2\qquad\hbox{for all }x,y\in K.\] The strong monotonicity property implies, in particular, that the following relation holds 
\[\max_{1\le j\le m}[F(x)-F(y)]_j\cdot [x-y]_j\geq \mu\|x-y\|_2^2\qquad \hbox{for all }x,y\in K,\]
where $[\cdot]_j$ denotes the $j$th element of the vector.
The preceding relation is the defining relation for the uniform P-function property of the mapping $F(\cdot)$. This property (weaker than the strong monotonicity of the mapping) guarantees the existence of a unique solution \textcolor{black}{to the VI$(K,F)$} according to Proposition~3.5.10 in~\cite{facchinei2003finite}, when the set $K$ has an appropriate Cartesian structure. 
}

\begin{definition}[Uniform P-function, Definition 2 \cite{parise2019variational}]\label{Def-P-function}
Given a set $K\subseteq \mathbb{R}^{m}$,
the mapping $F: K  \to \mathbb{R}^{m}$ is a uniform P-function on $K$ if there exists $\mu>0$ such that 
\[\max_{1\le j\le m}[F(x)-F(y)]_j\cdot [x-y]_j\geq \mu\|x-y\|_2^2\qquad \hbox{for all }x,y\in K,\]
where $[\cdot]_j$ denotes the $j$th element of the vector.
\end{definition}

It turns out that, if $F:K\to\mathbb{R}^{m}$ is a uniform P-function 
on the set $K$, then it is norm coercive on the set $K$, i.e.,
$\lim_{\|x\|_2\to\infty, x\in K}\|F(x)\|_2=+\infty$.

\begin{lemma}\label{Lem-P_matices-Aux-0}
Let mapping $F:K\to\mathbb{R}^{m}$ be a uniform P-function on the set $K$ {\color{black} with a constant $\mu>0$}. Then, the mapping $F(\cdot)$ is norm coercive over $K$.
%\[\lim_{\|x\|_2\to\infty, x\in K}\|F(x)\|_2=+\infty.\]
\end{lemma}
\begin{proof}
\textcolor{black}{By H\"older's inequality, we have for all $x,y\in K$, 
\[\|F(x)-F(y)\|_2\|x-y\|_2\ge \sum_{i=1}^{m}|
[F(x)-F(y)]_i[x-y]_i|. \]
Since $F(\cdot)$ is a uniform P-function, it follows that 
\begin{align*}
\sum_{i=1}^{m}\!|
[\!F(x)-F(y)\!]_i[x-y]_i| \ge \max_{1\le j\le m} [\!F(x)-F(y)\!]_j\cdot [x-y]_j\ge \mu\|x-y\|^2_2.
\end{align*}
Therefore,  
\[\|F(x)-F(y)\|_2\ge \mu\|x-y\|_2\qquad\hbox{for all $x,y\in K$},\]
implying that 
\[\|F(x)\|_2\ge \|F(x)-F(y)\|_2-\|F(y)\|_2\ge \mu\|x-y\|_2- \|F(y)\|_2.\]
By keeping $y\in K$ fixed and letting $\|x\|_2\to\infty$, with $x\in K$,
we obtain 
\[\lim_{\substack{\|x\|_2\to\infty, \ x\in K}}\|F(x)\|_2=+\infty.\] \hfill\hfill$\square$}
\end{proof}

Next, we consider a differentiable mapping $F(\cdot)$ with the uniform P-function property over a closed convex set $K$ and provide the implications of this property for the Jacobian $\nabla F(\cdot)$.

\begin{theorem}\label{thm-P-function}
    Let  $K\subseteq\mathbb{R}^m$ be a closed convex set and let $F:K\to\mathbb{R}^{m}$ be a uniform P-function on the set $K$ {\color{black} with a constant $\mu>0$}. Assume that the mapping $F(\cdot)$ is continuously differentiable over the set $K$. Then, we have: 
    \begin{itemize}
        \item [(a)] If $K=K_1\times\cdots \times K_m$, where each $K_i\subset\mathbb{R}$ is a closed convex set, then all principal sub-matrices of $\nabla F(x)$, for all $x\in K$, have singular values bounded away from zero by $\mu>0$.
        %,  where $\mu$ is the constant from the uniform P-function property.
        Also,
        all principal minors of the Jacobian $\nabla F(\cdot)$ are nonzero on the set $K$. 
        \item [(b)] If the set $K$ is contained in some open convex set $\mathcal{O}$ and $F(\cdot)$ is continuously differentiable over $\mathcal{O}$, then all principal sub-matrices of $\nabla F(x)$, for all $x\in K$, have singular values bounded away from zero by $\mu>0$. Consequently, all principal minors of the Jacobian $\nabla F(\cdot)$ are nonzero on the set $\mathcal{O}$. 
    \end{itemize}
\end{theorem}
\begin{proof}
(a)~Let $M\subseteq\{1,\ldots,m\}$ be an arbitrary 
nonempty subset of indices with the cardinality $r$. 
Consider an arbitrary point $x\in K$. 
Let $\nabla F_M(x)$ be the $r\times r$ sub-matrix obtained by deleting the $j$th row and $j$th column from the Jacobian matrix $\nabla F(x)$ for all $j\notin M$. 
Let $\sigma_{\min}(x)$ be the smallest singular value of $\nabla_M F(x)$, and let $u(x),v(x)\in\mathbb{R}^r$ be unit vectors such that 
 \[\nabla F_M(x)v(x)=\sigma_{\min}(x) u(x).\] 

Define the vector $w(x)\in\mathbb{R}^{m}$ as follows: 
\begin{equation}\label{eq-defwvec}
    w_i(x)=v_i(x)\quad\hbox{for $i\in M$}\qquad\hbox{and}\qquad w_i(x)=0 \qquad\hbox{for $i\notin M$}.
    \end{equation}

\def\e{\epsilon}
For the point $x\in K$ and the vector $w(x)$, due to the special Cartesian structure of the set $K$ (with one-dimensional convex sets), for some small enough $\epsilon>0$, we have 
$x+\e w(x)\in K$ or $x-\e w(x)\in K$.
Define the vector $y_\e$ as follows:
\[y_\e=\begin{cases}
    x+\e w(x) & \hbox{ if $x+\e w(x)\in K$},\cr
    x-\e w(x) & \hbox{ if $x+\e w(x)\notin K$, \ $x-\e w(x)\in K$}.
\end{cases}\]
By the mean value theorem, for each $i$, there is $s_i\in(0,1)$ such that 
 \[F_i(y_\e)-F_i(x)=\left\langle\nabla F_i\left(x+s_i(y_\e-x)\right),y_\e-x\right\rangle.\] 
 Since $x,y_\e\in K$, $s_i\in(0,1)$ for all $i=1,\ldots,m$,  and the set $K$ is convex, it follows that $x+s_i(y_\e-x) \in K$ for all $i$,
 so the Jacobian $\nabla F_i\left(x+s_i(y_\e-x)\right)$ is defined.
 Stacking $F_i(y_\e)-F_i(x)$, $i=1,\ldots,m$, in a column vector
 we can write 
 \[F(y_\e)-F(x)=\nabla F_{\bf{x}(\e)} (y_\e-x),\] 
 where $\nabla F_{\bf{x}(\e)}$ is the matrix with rows $\nabla F_i\left(x+s_i(y_\e-x)\right)$, $i=1,\ldots,m$. 
Since the mapping $F(\cdot)$ satisfies the uniform P-function property on the set $K$, we have that
\[\max_{1\leq j \leq m}
[F(y_\e)-F(x)]_j\cdot [y_\e-x]_j\ge \mu \|y_\e-x\|_2^2.\]
By the definition of $y_\e$, we have $y_\e-x=w(x)$ and, since $\|w(x)\|_2=1$ (see~\eqref{eq-defwvec}), it follows that 
\[\e^2 \max_{1\leq j \leq m}\left[\nabla F_{{\bf x}(\e)}w(x)\right]_j
\cdot w_j(x)\geq \mu\e^2.\] 
Letting $\e\to0$, and noting that $\nabla F_{{\bf x}(\e)}\to \nabla F(x)$, we obtain that 
\[\max_{1\leq j \leq m} \left[\nabla F(x)w(x)\right]_j
\cdot w_j(x)\geq \mu.\]

By the definition of the vector $w(x)$ in~\eqref{eq-defwvec}, we have that 
$w_j(x)=0$ for $j\not\in M$, so it follows that
\[\sum_{i\in M} \left|\left[\nabla F(x)w(x)\right]_i \cdot w_i(x) \right|\ge 
\max_{j\in M} \left[\nabla F(x)w(x)\right]_j
\cdot w_j(x)\geq \mu.\]
Further, by the definition of the vector $w(x)$, we have that 
\[w_i(x)=v_i(x),\quad [\nabla F(x)w(x)]_i=[\nabla F_M(x)v(x)]_i\qquad\hbox{for all }i\in M.\]
By combining the preceding two relations, we obtain that 
\[%\langle v(x), \nabla F_M(x)v(x)\rangle =
\sum_{i\in M} \left|\left[\nabla F_M(x)v(x)\right]_i \cdot v_i(x) \right|\ge \mu.\]
Since $\nabla F_M(x)v(x)=\sigma_{\min}(x) u(x)$ and $\sigma_{\min}(x)\ge0$, it follows that 
\[\sigma_{\min}(x) \, \sum_{i\in M} |u_i(x)|\, |v_i(x)| \ge \mu.\]
%Note that we cannot have $\sigma_{\min}(x)=0$ since 
%$\mu>0$. Thus, $\sigma_{\min}(x)>0$.
By using the Cauchy-Schwarz inequality and $\|u(x)\|_2=1$, $\|v(x)\|_2=1$, from the preceding inequality we obtain $\sigma_{\min}(x)\ge \mu$. 
Thus,
all the singular values of the matrix $\nabla F_M(x)$ are uniformly bounded away from zero on the set $K$. As a consequence $\det(\nabla F_M(x))\ne0$.

\noindent
(b)~The proof is along the lines of the proof of part (a) with minor adjustments. We replace the set $K$ with the open convex set $\mathcal{O}$. We let $y_\e=x+\e w(x)\in\mathcal{O}$ for some small enough $\e>0$ so that $y_\e\in \mathcal{O}$, which is possible  since $\mathcal{O}$ is an open set. Hence, $x,y_\e\in \mathcal{O}$, $s_i\in(0,1)$ for all $i$, which by the convexity of the set $\mathcal{O}$ implies that  
$x+s_i(y_\e-x) \in \mathcal{O}$ for all $i$.
The rest of proof follows the same line of argument as in the part (a).
\hfill $\square$
\end{proof}
%--------------------------------------------------------------
\subsection{Non-Singularity of Jacobian}\label{subs-nonsing}
%--------------------------------------------------------------
The non-singularity condition on the Jacobian can be verified by checking the positive definiteness of the matrices
$\nabla{F(x)} \nabla{F(x)}^\top$ for all $x$.
In particular, by Lemma 3 in~\cite{NonMonotoneVI}, we have that $\det\left(\nabla{F(x)}\right) \neq 0$ if and only if the matrix $\nabla{F(x)} \nabla{F(x)}^\top$ is positive definite.

{\color{black} As we will see, another condition that guarantees the non-singularity of the Jacobian $\nabla F(\cdot)$ is the uniform P-matrix condition~\cite{parise2019variational}, the definition of which is somewhat involved. The definition has been originally proposed in~\cite{parise2019variational} (Definition~D.1), which is motivated by the P-matrix concept, 
stating that: a matrix $A$ is a P-matrix if all
its principal minors have positive determinant, which is equivalent to (by Theorem~3.3 in~\cite{fiedler-ptak-Pmatrices}): for any nonzero $w\in\mathbb{R}^m$, there exists a positive definite diagonal matrix $H_w$ such that 
\[\langle w, H_w A w\rangle>0.\]
The P-matrices have been used in connection to properties of the Jacobian $\nabla F(\cdot)$ to investigate the existence of solutions to Cartesian VI($K,F$), where the set $K$ is a Cartesian product of sets (see Section~3.5 in~\cite{facchinei2003finite}).
When requiring that the Jacobian $\nabla F(x)$ is a P-matrix for all $x\in K$,
the P-matrix condition extends to the following condition:
for any nonzero $w\in\mathbb{R}^m$ and any $x\in K$, there exists a positive definite diagonal matrix $H_{w,x}$ such that 
\[\langle w, H_{w,x} \nabla F(x) w\rangle>0.\]
Strengthening this condition further to have $\langle w, H_{w,x} \nabla F(x) w\rangle\ge \eta\|w\|_2^2$, for some $\eta>0$, and allowing for the Jacobian $\nabla F(x)$ to be evaluated at different points for different coordinates lead to the following definition.
In the definition, for a matrix $A$, we use $[A]_{i:}$ to denote the $i$-the row of the matrix.}

\begin{definition}[Uniform P-matrix Condition\footnote{This definition is a slight modification of Definition~D.1 in~\cite{parise2019variational} where the set $K$ is taken to be the entire $\mathbb{R}^m$ space.}, Definition~D.1 \cite{parise2019variational}]\label{Def-P matrix}
The mapping $x\mapsto A(x)\in \mathbb{R}^{m\times m}$ satisfies a  uniform P-matrix condition on a set $K\subseteq\mathbb{R}^{m}$ 
if  there exists $\eta> 0$ such that for any $w\in\mathbb{R}^m$ and any collection ${\bf x}_m=\{x^i,i=1,\ldots,m\}$ of $m$ points in  the set $K$, there is a diagonal matrix $H_{w,{\bf x}_m}\succ 0$ such that 
\[\langle w, H_{w,{\bf x}_m} A_{{\bf x}_m} w\rangle \geq \eta \|w\|_2^2\]
and {\color{black} $\sup_{w\in\mathbb{R}^m} \sup_{z\in K^m} \lambda_{\max}(H_{w,z}) < \infty$},
where $A_{{\bf x}_m}$ is the matrix constructed from the matrices $A(x^1),\ldots,A(x^m)$, 
%with $x^i\in {\bf x_m}$, 
by taking the $i$-th row of the matrix $A(x^i)$ and placing it in the $i$-th row of $A_{{\bf x}_m}$ for all $i=1,\ldots,m,$ i.e.,
\[A_{{\bf x}_m} =\left[\begin{array}{c}
[A(x^1)]_{1:}\cr\vdots\cr
[A(x^m)]_{m:}
\end{array}\right].\]
%with $[A(x^i)]_{i:}$ denoting the $i$-th row of the matrix $A(x^i)$, with $x^i\in {\bf x}_m$, for all $i=1,\ldots,m$.
\end{definition}

An implication of this definition is that, if $\nabla F(\cdot)$
satisfies the uniform P-matrix condition on the set $K$, then the mapping  $F(\cdot)$ is  a uniform P-function on $K$, shown in Proposition~3 c)
in~\cite{parise2019variational}.
}

When a mapping $x\mapsto A(x)\in \mathbb{R}^{m\times m}$ satisfies the uniform P-matrix condition, all singular values of principal sub-matrices of $A(x)$ have a uniform lower bound, as shown in the following result.

\begin{theorem}\label{Thm-Singval-Pmat}
Let a mapping $x\mapsto A(x)\in \mathbb{R}^{m\times m}$ satisfy the uniform P-matrix condition 
over a set $K\subseteq\mathbb{R}^{m}$. Then, all principal sub-matrices of $A(x)$ have singular values uniformly bounded away from zero on the set $K$. In particular, all principal minors\footnote{Given an $m\times m$ matrix $A$, an $r\times r$ principal sub-matrix of $A$ is the matrix obtained by taking the entries of $A$ that lie in the same set of $r$ rows and columns. An $r\times r$ principal minor of $A$ is the determinant of an $r\times r$ principal sub-matrix of $A$; see page 494 of \cite{meyer2023matrix}.} of $A(x)$ are nonzero for all $x\in K$.
\end{theorem}
\begin{proof}
Consider an arbitrary point $x\in K$ and an arbitrary 
nonempty subset $M\subseteq\{1,\ldots,m\}$ of indices with cardinality $r$.
Consider the $r\times r$ matrix $A_M(x)$ obtained by deleting the $j$th row and $j$th column from the matrix $A(x)$ for all $j\notin M$. 
Let $\sigma_M(x)$ be an arbitrary singular value of the matrix $A_M(x)$. By the definition of the singular value there are two unit vectors $u(x),v(x)\in \mathbb{R}^{r}$ such that $A_M(x)v(x)=\sigma_M(x)u(x)$. 

Define the vector $w(x)\in\mathbb{R}^m$ with entries 
\[w_i(x)=v_i(x)\quad\hbox{for $i\in M$}\qquad\hbox{and}\qquad w_i(x)=0 \qquad\hbox{for $i\notin M$}.\]
Since the mapping $x\mapsto A(x)$ satisfies the uniform P-matrix condition 
over $K$,
  by Definition~\ref{Def-P matrix}, with ${\bf x_m}$ consisting of $m$ copies of the vector $x$ ($x^i=x$ for all $i$), we have $A_{{\bf x}_m}= A(x)$ and 
  \[\langle w(x), H_{w(x),x}A(x)w(x)\rangle \geq \eta \|w(x)\|_2^2,\]
  where $H_{w(x),x}$ is a positive definite diagonal matrix and $\eta>0$ is independent of $x$.
By the construction of the vector $w(x)$, we have that $\|w(x)\|_2=\|v(x)\|_2=1$, and the preceding relation is equivalent to
  \[\sum_{i\in M} v_i(x)\cdot [H_{w(x),x}A(x)w(x)]_i \geq \eta.\]
Using  the Cauchy-Schwarz inequality and the fact that $v(x)$ is a unit vector,
we obtain
\[\sqrt{\sum_{i\in M}  |[H_{w(x),x}A(x)w(x)]_i|^2}\geq \eta,\]
implying that 
\[\max_{i\in M} [H_{w(x),x}]_{ii}
\sqrt{\sum_{i\in M}  |[A(x)w(x)]_i|^2}\geq \eta,\]
where we use the fact that the matrix $H_{w(x),x}$ is diagonal and positive definite.
Since $A_M(x)v(x)=\sigma_M(x)u(x)$, we have
\[[A(x)w(x)]_i=[A_M(x)v(x)]_i=\sigma_M(x)u_i(x)\qquad\hbox{for all $i\in M$}.\]
Therefore, from the preceding relation, we obtain
\[\max_{i\in M} [H_{w(x),x}]_{ii}\cdot\sigma_M(x) 
\sqrt{\sum_{i\in M} u_i^2}\geq \eta.\]
Since $\|u\|_2=1$ and
$\max_w \max_{z\in K^m} \lambda_{\max}(H_{w,z}) < \infty$ by the uniform P-matrix condition, it follows that 
\[\sigma_M(x)\ge \frac{\eta}{C}\qquad\hbox{where} \quad C=\max_w \max_{z\in K^m} \lambda_{\max}(H_{w,z}),\]
showing that
a uniform lower bound on any singular value $\sigma_M(x)$ of the sub-matrix $A_M(x)$ over the set $K$ exists.  Since the singular values $\sigma_M(x)$ of $A_M(x)$ are nonzero over the set $K$, it follows that $\det(A_M(x))\ne0$ for all $x\in K$ and any nonempty index set $M\subseteq\{1,\ldots,m\}$.
    \hfill $\square$
\end{proof}

%--------------------------------------------------------
\section{Non-Monotone {\color{black} Variational Inequalities}}\label{Sec-Main result-VI}
%--------------------------------------------------------
In this section, we extend the existence result of Theorem~\ref{Thm-main theorem-prime} to the general
VI$(K,F)$, where the mapping $F(\cdot)$ is not necessarily monotone. To do so,
we work with the normal mapping $F^{\rm nor}_K(\cdot)$ associated with the VI and its generalized Jacobian $\partial F^{\rm nor}_K(\cdot)$. 
In Subsection~\ref{ssec-main}, we provide sufficient conditions for the existence of solutions to non-monotone VI$(K,F)$, 
%in Theorem~\ref{thm-Normal-Mapping_Sol_Existence0}. We then provide an alternative sufficient condition for the existence of solutions in  Theorem~\ref{thm-Normal-Mapping_Sol_Existence}, 
assuming that the normal mapping $F^{\rm nor}_K(\cdot)$ is norm coercive and some non-singularity properties of the generalized Jacobian $\partial F^{\rm nor}_K(\cdot)$ hold, such as a maximal rank. In Subsections~\ref{ssec-normcoercivity} and~\ref{ssec-gen-jac-normal-map}, respectively, we elaborate on these two assumptions.

%%%%%%%%%%%%%%%%%%%%%%%%%%%%%%%%%%%%%%%%%%%%%%%%%%%%%%%%%%%%%%%%%
\subsection{Sufficient Conditions for Existence of Solutions}\label{ssec-main}
%%%%%%%%%%%%%%%%%%%%%%%%%%%%%%%%%%%%%%%%%%%%%%%%%%%%%%%%%%%%%%%%%
\def\argmin{{\mathop{\rm argmin}}}
\textcolor{black}{For a (nonempty) closed convex set $K\subset\mathbb{R}^m$, the projection mapping $\Pi_K[\cdot]:\mathbb{R}^m\to K$ is defined by $\Pi_K[x]:=\argmin_{z\in K}\|x-z\|$ for all $x\in\mathbb{R}^m$. The normal mapping $F_K^{\rm nor}(\cdot)$ associated with a VI$(K,F)$} is defined as follows (\cite{facchinei2003finite}, page 83): 
\[F_K^{\rm nor}(x)=x-\Pi_K\left[x\right]+F(\Pi_K\left[x\right])
\qquad\hbox{for all $x\in\mathbb{R}^{m}$}.\]
The following classical result provides a characterization of the solution set of a VI$(K,F)$ via normal mapping $F_K^{\rm nor}(\cdot)$.

\begin{theorem}[Proposition 1.5.9 \cite{facchinei2003finite}]\label{Thm-auxiliary-normal map solution}
    Let $K\subseteq \mathbb{R}^{m}$ be closed convex and $F:K\to \mathbb{R}^{m}$ be arbitrary. A vector $x^*$ belongs to {\rm SOL}$(K,F)$ if and only if there is a vector $v$ such that $x^*=\Pi_K[v]$ and $F_K^{\rm nor}(v)=0$.
\end{theorem} 

The normal mapping $F_K^{\rm nor}(\cdot)$ associated with a VI$(K,F)$ need not be differentiable even when the VI mapping $F(\cdot)$ is. To deal with this,
we will work with a concept of generalized Jacobian of a mapping.
\begin{definition}[Generalized Jacobian, Definition 1 \cite{clarke1976inverse}] \label{Def-Generalized Jacobian}
The generalized Jacobian of a mapping $F(\cdot)$ at a point $x\in \mathbb{R}^{m}$, denoted by $\partial F(x)$, is the convex hull of all matrices obtained in the limit
$\lim_{k\rightarrow \infty} \nabla F(x^k)$ along sequences $\{x^k\}$ converging to $x$
such that $F(\cdot)$ is differentiable at $x^k$ for all~$k$, 
\begin{align*}
\partial F(x)={\rm Conv}\left\{M\mid M=\lim_{k\rightarrow \infty} \nabla F(x^k),
\ \lim_{k\to\infty}x^k=x, \ \nabla F(x^k) \hbox{ exists for all~$k$}\right\},
\end{align*}
where $\text{Conv}(X)$ denotes the convex hull of a set $X$.
\end{definition}

The following classical result characterizes the generalized Jacobian $\partial F(\cdot)$ for locally Lipschitz mapping $F(\cdot)$, where  {\it $F(\cdot)$ being locally Lipschitz at a point $x$} means that $F(\cdot)$ is Lipschitz continuous in a neighborhood of the point $x$.
\begin{theorem}[Generalized Jacobian
Properties, Proposition 1 \cite{clarke1976inverse}]\label{Thm-Proposition 1 of Clark paper}
    Let mapping $F(\cdot)$ be locally Lipschitz at a point $x\in \mathbb{R}^m$. Then, the generalized Jacobian $\partial F(x)$ is a nonempty, compact, and convex set in the space of all $m\times m$ matrices equipped with norm $\|M\|=\max_{1\le i,j\le m}|m_{i,j}|$.
\end{theorem}

We now state our main solution existence result for a VI$(K,F)$.

{\color{black} 
\begin{theorem}\label{thm-Normal-Mapping_Sol_Existence0}
Let $K\subseteq \mathbb{R}^{m}$ be a nonempty closed convex set, the mapping $F:K\to\mathbb{R}^m$ be locally Lipschitz continuous on $K$, and the normal mapping $F^{\rm nor}_K(\cdot)$ be norm coercive. Assume 
that, at every point $x\in\mathbb{R}^m$ where $F_K^{\rm nor}(x)\neq 0$,
we have 
 \[0\notin \left\{M^\top F^{\rm nor}_K(x)\mid M\in\partial F^{\rm nor}_K(x)\right\}.\] 
Then, the solution set {\rm SOL}$(K,F)$ is non-empty and compact.
\end{theorem}

\begin{proof} 
Consider the following optimization problem
\begin{equation}\label{eq-optpro}
\inf_{x\in \mathbb{R}^{m}}\|F_K^{\rm nor}(x)\|_2^2.
\end{equation}
    By the continuity of the projection mapping $\Pi_K[\cdot]$ and the assumed continuity of the mapping $F(\cdot)$,
    we have that $F_K^{\rm nor}(v)=v-\Pi_K\left[v\right]+F(\Pi_K\left[v\right])$ is continuous. 
     By the norm coercivity of $F_K^{\rm nor}(\cdot)$, the objective function of the problem is coercive.
    Hence, by the Weierstrass Theorem (Proposition~2.1.1 in~\cite{bno2003convex}), the problem has a finite optimal value and its solution set is nonempty and compact, i.e.,   
    \[X^*=\left\{x\in\mathbb{R}^m\mid \|F_K^{\rm nor}(\bar x)\|_2^2= b\right\}\ne\emptyset \ \hbox{compact},\qquad b=\min_{x\in \mathbb{R}^{m}}\|F_K^{\rm nor}(x)\|_2^2.\]
    
    We want to prove that $b=0$, which allows us to relate the set $X^*$ to the solutions of VI$(K,F)$.
    To arrive at a contradiction, assume that $b>0$. Thus, 
    \begin{equation}\label{eq-notzero}
        F_K^{\rm nor}(\bar x)\ne0\qquad\hbox{for any $\bar x\in X^*$}. 
    \end{equation}
    Since $F(\cdot)$ is locally Lipschitz continuous over the set $K$, the normal mapping $F_K^{\rm nor}(\cdot)$ is also locally Lipschitz continuous and, by Theorem~\ref{Thm-Proposition 1 of Clark paper}, the generalized Jacobian $\partial F_K^{\rm nor}(x)$ is nonempty, compact, and convex  for all $x\in\mathbb{R}^m$. 
    %In particular, for any $\bar x\in X^*$, the generalized Jacobian
    %$\partial F_K^{\rm nor}(\bar x)$ is nonempty, compact, and convex.
    
    Let $\bar x\in X^*$ be arbitrary.
    Since $\bar x\in X^*$ solves the (unconstrained) minimization problem~\eqref{eq-optpro}, by the necessary optimality conditions  of Proposition~7.1.12 in~\cite{facchinei2003finite}, it follows that 
    \begin{equation}\label{eq-optim1}
    0\in\partial \|F^{\rm nor}_K(\bar x)\|^2_2.
    \end{equation}
    The norm function is continuously differentiable, so by the Clarke's chain rule (Proposition~7.1.11 in~\cite{facchinei2003finite}), we have that 
    \begin{equation}\label{eq-optim2}
    \partial \|F^{\rm nor}_K(\bar x)\|^2_2
    =\left\{2M^\top F^{\rm nor}_K(\bar x)\mid M\in\partial F^{\rm nor}_K(\bar x)\right\}.
    \end{equation}
    Relations~\eqref{eq-optim1} and \eqref{eq-optim2} imply that 
     \[0\in \left\{M^\top F^{\rm nor}_K(\bar x)\mid M\in\partial F^{\rm nor}_K(\bar x)\right\},\]
     while $F_K^{\rm nor}(\bar x)\ne0$ by relation~\eqref{eq-notzero}. This contradicts our assumption that 
     \[0\notin \{M^\top F^{\rm nor}_K(\bar x)\mid M\in\partial F^{\rm nor}_K(\bar x)\}\qquad \hbox{when 
     $F_K^{\rm nor}(\bar x)\ne0$}.\]
     Hence, we must have $b=0$ and $F_K^{\rm nor}(\bar x)=0$ for all $\bar x\in X^*$. By Theorem~\ref{Thm-auxiliary-normal map solution}, it follows
    that ${\rm SOL}(K,F)=\{\Pi_K[\bar x]\mid \bar x\in X^*\}$. The solution set SOL$(K,F)$ is nonempty since $X^*\ne\emptyset$. It is compact since $X^*$ is compact and the projection mapping is continuous.
    \hfill $\square$
\end{proof}

%Next, we provide an example showing that under the conditions of Theorem~\ref{thm-Normal-Mapping_Sol_Existence0}, a VI need not have a unique solution.
%\begin{example}\label{ex-multiplesols}
%\end{example}

In the next example, we illustrate a case when the mapping $F(\cdot)$ is not a uniform P-function on $K$ and, thus, Proposition 3.5.10(b) of~\cite{facchinei2003finite} cannot be applied to assert the existence of solutions, while Theorem~\ref{thm-Normal-Mapping_Sol_Existence0} can be used.

\begin{example}\label{example-vi} Consider the set $K=K_1\times K_2$, with $K_1=K_2=[0,\infty)$, and the linear mapping $F(x)=[x_1+2x_2,3x_1+x_2]^\top$.  
The unique solution to VI$(K,F)$ is $x^*=[0,0]^\top$. 
The mapping $F(\cdot)$ is continuously differentiable.
    We will show that the normal mapping $F_K^{\rm nor}(\cdot)$ satisfies the conditions of Theorem~\ref{thm-Normal-Mapping_Sol_Existence0}.

The normal mapping $F_K^{\rm nor}(\cdot)$ is given by 
\begin{align}
    F_K^{\rm nor}(x) &=x - \Pi_K[x] + F(\Pi_K[x])=\begin{bmatrix}
x_1 \\
x_2
\end{bmatrix}-\begin{bmatrix}
\Pi_{K_1}[x_1] \\
\Pi_{K_2}[x_2]
\end{bmatrix}+\begin{bmatrix}
\Pi_{K_1}[x_1]+2\Pi_{K_2}[x_2] \\
3\Pi_{K_1}[x_1]+\Pi_{K_2}[x_2]
\end{bmatrix} \nonumber \\
%&=\begin{bmatrix}
%x_1 + 2\Pi_{K_2}[x_2]\\[1mm]
%3\Pi_{K_1}[x_1] +x_2
%\end{bmatrix}\nonumber\\
&=\begin{bmatrix}
x_1  + 2\max\{x_2,0\}\\[1mm]
3\max\{x_1,0\} +x_2
\end{bmatrix}.\nonumber
\end{align}
The mapping $F_K^{\rm nor}(\cdot)$ is piecewise linear, given by
\[
\!F_K^{\rm nor}\!(x) \!=\!
\begin{cases}
[x_1 + 2x_2,\; 3x_1 +x_2]^\top & x\in R_1\!=\!\{x\in\mathbb{R}^2\mid  x_1 \ge 0,\;\! x_2 \ge 0\},\\[1mm]
[x_1,\;\! 3x_1+x_2]^\top & x\in R_2\!=\!
\{x\in\mathbb{R}^2\mid  x_1 \ge 0,\; x_2 \le 0\},\\[1mm]
[x_1+ 2x_2,\; x_2]^\top & x\in R_3\!=\! 
\{x\in\mathbb{R}^2\mid  x_1 \le 0,\;\! x_2 \ge 0\},\\[1mm]
[x_1,\; x_2]^\top&  x\in R_4\!=\!\{x\in\mathbb{R}^2\mid  x_1 \le 0,\;\! x_2 \le 0\}.
\end{cases}
\]
Figure~\ref{fig:example-normal} depicts the sets $R_i$ and corresponding expressions for the normal mapping for visualization.
\begin{figure}[h!]
    \centering
    \includegraphics[width=0.7\linewidth]{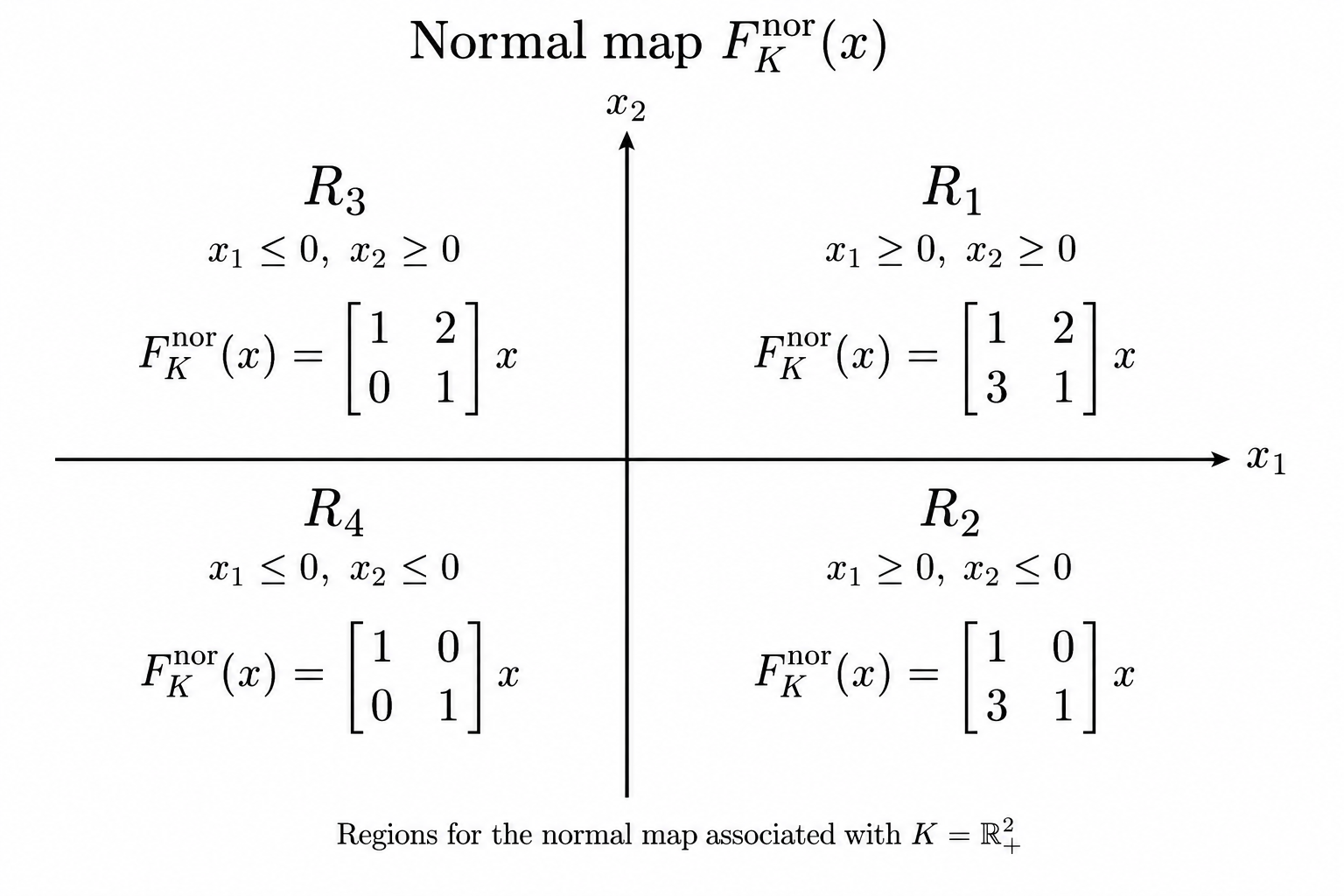}
    \caption{The normal mapping $F_K^{\rm nor}(x) =A_i x$ for $x\in R_i$, $i=1,2,3,4$,
where $A_i$ is the matrix corresponding to the linear mapping on the set $R_i$. (The plot has been generated using ChatGPT.)}
    \label{fig:example-normal}
\end{figure}

For each $i$, we have 
\[\|F_K^{\rm nor}(x)\|_2=\|A_ix\|_2\geq\sigma_{\min}(A_i)\|x\|_2.\] As long as $\sigma_{\min}(A_i)>0$ for all $i$, the mapping $F_K^{\rm nor}(\cdot)$ will be norm coercive. 
Noting that $A_4$ is the identity matrix, we have that $\sigma_{\min}(A_4)=1$. For the other cases, we have
\[\sigma_{\min}(A_1)=\sqrt{\frac{15-5\sqrt{5}}{2}}, \ 
\sigma_{\min}(A_2)=\sqrt{\frac{11 - 3\sqrt{13}}{2}},
 \ \sigma_{\min}(A_3)=\sqrt{3 - 2\sqrt{2}}.\]
Thus, for each $i$, the smallest singular value of $A_i$ is positive.

We next consider the generalized Jacobian of the normal mapping $F_K^{\rm nor}(\cdot)$ to determine if it satisfies the condition of Theorem~\ref{thm-Normal-Mapping_Sol_Existence0}.
We note that the only zero of the normal mapping $F_K^{\rm nor}(\cdot)$ is the zero vector, i.e.,
\[F_K^{\rm nor}(x)\ne 0\qquad\hbox{for all $x\ne0$}.\]
The generalized Jacobian of $F_K^{\rm nor}(\cdot)$
coincides with the Jacobian of the linear map $x\mapsto A_ix$ for $x$ in the interior of the set $R_i$, for all $i=1,2,3,4.$ Since $A_i$ is non-singular for each $i$, it follows that
\[[\nabla F_K^{\rm nor}(x)]^\top F_K^{\rm nor}(x)\ne0
\quad \hbox{for all $x\in R_i^\circ$, $i=1,2,3,4,$}\]
with $X^\circ$ denoting the interior of a set $X$. Hence,
the condition on the generalized Jacobian of Theorem~\ref{thm-Normal-Mapping_Sol_Existence0} is satisfied in the interior of the sets $R_i$. 

It remains to consider the set $\{M^\top F_K^{\rm nor}(x)\mid M\in \partial F_K^{\rm nor}(x)\}$ at the boundary points of the sets $R_i$, excluding the zero vector. 
By Definition~\ref{Def-Generalized Jacobian}, the generalized Jacobian 
$\partial F_K^{\rm nor}(x)$ at a boundary point $x$ of a set $R_i$ is given by the convex hull of the affine mappings $A_i$ associated with sets $R_i$ and $R_j$ sharing the boundary point~$x$. 
Thus, we have  
\[
\partial F_K^{\rm nor}\!(x) \!=\!
\begin{cases}
B_{12}(\alpha_1)=\begin{bmatrix}
1 \quad&& 2\alpha_1\\[1mm]
3 \quad && 1
 \end{bmatrix}
& \ \hbox{for } x\in (R_1\cap R_2)\setminus\{0\}, \alpha_1\in[0,1],\\[2mm]
B_{13}(\alpha_2)= \begin{bmatrix}
1 \quad&& 2\\[1mm]
3\alpha_2 \quad && 1
 \end{bmatrix}
& \ \hbox{for } x\in (R_1\cap R_3)\setminus\{0\}, 
%x_1 = 0, x_2 > 0, 
\alpha_2\in[0,1],\\[2mm]
B_{24}(\alpha_3)=\begin{bmatrix}
1 \quad&& 0\\[1mm]
3\alpha_3  \quad && 1
 \end{bmatrix}
& \ \hbox{for } x\in (R_2\cap R_4)\setminus\{0\}, \alpha_3\in[0,1],\\[2mm]
B_{34}(\alpha_4)=\begin{bmatrix}
1 \quad&& 2\alpha_4 \\[1mm]
0\quad && 1
 \end{bmatrix}
& \ \hbox{for } x\in (R_3\cap R_4)\setminus\{0\}, \alpha_4\in[0,1],
\end{cases}
\]
where the subscripts $ij$ for the matrix $B$ indicate the indices of the sets $R_i$ and $R_j$ that have a common boundary.
The determinants of the matrices $B_{24}(\alpha_3)$ and $B_{34}(\alpha_4)$ are
\[\det(B_{24}(\alpha_3))=1, \ \alpha_3\in[0,1],\qquad
\det(B_{34}(\alpha_4))=1, \ \alpha_4\in[0,1].\]
Thus, they have a full rank for all $\alpha_3,\alpha_4\in[0,1]$. Since $F_K^{\rm nor}(x)\ne0$ for $x\ne0$, it follows that  
\[0\not\in\left\{M^\top F_K^{\rm nor}(x)\mid M\in \partial  F_K^{\rm nor}(x)\right\}\qquad\hbox{for $x\in (R_2\cap R_4)\cup( R_3\cap R_4)$, $x\ne0$},\]
and the condition of Theorem~\ref{thm-Normal-Mapping_Sol_Existence0} is satisfied at these points.

Finally, we consider  the points $x\in (R_1\cap R_2)\cup(R_1\cap R_3)$, with $x\ne0$, which involve the matrices $B_{12}(\alpha_1)$ and $B_{13}(\alpha_2)$. Checking the determinants of these matrices shows that they can be zero for certain values of $\alpha_1$ and $\alpha_2$, and we cannot threat them as in the previous case. In this case, we determine the whole set $\{M^\top F_K^{\rm nor}(x)\mid M\in \partial F_K^{\rm nor}(x)\}$, which we denote by 
\[(\partial\! F_K^{\rm nor}(x))^\top F_K^{\rm nor}(x):=\left\{M^\top F_K^{\rm nor}(x)\mid M\in \partial F_K^{\rm nor}(x)\right\}.\]
We start with the normal mapping, for which we have
\[F_K^{\rm nor}\!(x) \!=\!
\begin{cases}
x_1[1, 3]^\top
& \ \hbox{for $x\in (R_1\cap R_2)\setminus\{0\}=\{x\in\mathbb{R}^2\mid x_1 >0, x_2=0\}$},\\[2mm]
x_2[2, 1]^\top
& \ \hbox{for $x\in R_1\cap R_3=\{x\in\mathbb{R}^2\mid x_1=0, x_2 > 0\}$}.
\end{cases}
\]
Using the expressions for $\partial F_K^{\rm nor}\!(x)$ for such points $x$, we find that
\[(\partial F_K^{\rm nor}(x))^\top F_K^{\rm nor}\!(x) \!=\!
\begin{cases}
x_1[10, 2\alpha_1+3]^\top
&  \hbox{for $x_1 >0, x_2=0, \alpha_1\in[0,1]$},\\[2mm]
x_2[2+3\alpha_2, 5]^\top
&  \hbox{for $x_1=0, x_2 > 0, \alpha_2\in[0,1].$}
\end{cases}
\]
As we can see, the zero vector is not included in the set $(\partial F_K^{\rm nor}\!(x))^\top F_K^{\rm nor}\!(x)$, as long as $x\ne0$,
and the condition of Theorem~\ref{thm-Normal-Mapping_Sol_Existence0} is also satisfied in this case.

In conclusion, 
\[0\notin\left\{M^\top F_K^{\rm nor}(x)\mid M\in \partial F_K^{\rm nor}(x)\right\}\qquad\hbox{for all $x\ne0$},\]
showing that the generalized Jacobian condition 
of Theorem~\ref{thm-Normal-Mapping_Sol_Existence0} is satisfied for all $x\ne 0$, with $x=0$ being the unique point where $F_K^{\rm nor}\!(x)=0$.
Therefore, by Theorem~\ref{thm-Normal-Mapping_Sol_Existence0}, the VI$(K,F)$ must have a solution. \\
\noindent
{\it Uniform P-function property fails.} The mapping $F(\cdot)$ is a uniform P-function when the following relation holds for some $\mu>0$,
\[\max_{j=1,2}[F(x)-F(y)]_j\cdot [x-y]_j\geq \mu\|x-y\|_2^2\qquad \hbox{for all }x,y\in K.\]
Recalling that $F(x)=[x_1+2x_2,3x_1+x_2]^\top$, and using the fact that the mapping is linear, we see that 
\[F(x)-F(y)=\left[\begin{matrix}
x_1-y_1+2(x_2-y_2)\cr
3(x_1-y_1)+x_2-y_2
\end{matrix}\right].\]
Letting $w_i=x_i-y_i$ for $i=1,2$, 
the uniform P-function property reduces to 
\begin{equation}\label{eq-nop}
    \max\left\{w_1^2 +2w_1w_2,3w_1w_2+ w_2^2\right\}\geq \mu\|w\|_2^2\qquad \hbox{for all }w\in \mathbb{R}^2.
\end{equation}
Choosing $w_1\ne0$ and $w_2=-w_1$, the left hand side of relation~\eqref{eq-nop} is given by $\max\left\{-w_1^2,-2w_1^2\right\}$, which is negative since $w_1\ne0$. Hence, relation~\eqref{eq-nop} cannot hold, implying that
the game mapping $F(\cdot)$ is not a uniform P-function. Consequently, Proposition 3.5.10(b) of~\cite{facchinei2003finite} cannot
be used to claim the existence of a solution to VI($K,F$).\hfill $\square$
\end{example}

Next, we provide a stronger condition that implies that the condition on the generalized Jacobian of Theorem~\ref{thm-Normal-Mapping_Sol_Existence0} is satisfied. This stronger condition relies on  the notion of a generalized Jacobian of maximal rank.
\begin{definition}[Generalized Jacobian/Maximal Rank, Definition 2 \cite{clarke1976inverse}] \label{Def-fullrank Generalized Jacobian}
A generalized Jacobian
$\partial F(x)$ is said to be of maximal rank if every matrix $M$ in the definition of $\partial F(x)$ (Definition~\ref{Def-Generalized Jacobian}) has a full rank.
\end{definition}

The condition on the generalized Jacobian $\partial F^{\rm nor}_K(x)$ in Theorem~\ref{thm-Normal-Mapping_Sol_Existence0}  
is satisfied  when $\partial F^{\rm nor}_K(x)$ is of maximal rank, as seen in the following lemma.
\begin{lemma}\label{lem-ful-rank}
Let $K\subseteq \mathbb{R}^{m}$ be a nonempty closed convex set and the mapping $F:K\to\mathbb{R}^m$ be locally Lipschitz continuous on $K$.
If generalized Jacobian $\partial F^{\rm nor}_K(\bar x)$ is of maximal rank at some point $\bar x\in\mathbb{R}^m$ where $F^{\rm nor}_K(\bar x)\ne 0$, then
 \[0\notin \{M^\top F^{\rm nor}_K(\bar x)\mid M\in\partial F^{\rm nor}_K(\bar x)\}.\] 
\end{lemma}
\begin{proof}
%Under the assumptions on the set $K$ and the mapping $F(\cdot)$, the normal mapping $F_K^{\rm nor}(x)$ is locally Lipschitz. By Theorem~\ref{Thm-Proposition 1 of Clark paper}, the generalized Jacobian $\partial F_K^{\rm nor}(x)$ is nonempty for all $x\in\mathbb{R}^m$.
To prove the statement, we argue by contraposition. Suppose that there is a matrix $M\in\partial F^{\rm nor}_K(\bar x)$ such that 
\[M^\top F^{\rm nor}_K(\bar x)=0.\]
Since $F^{\rm nor}_K(\bar x)\ne 0$, it follows that $F^{\rm nor}_K(\bar x)$ is an eigenvector of $M^\top$ associated with the zero eigenvalue. Therefore, $M^\top$ is rank deficient and so is $M$.
Hence, $M$ does not have full rank, implying that $\partial F^{\rm nor}_K(\bar x)$ is not of maximal rank.
\hfill$\square$
\end{proof}
}

{\color{black}
As a consequence of Theorem~\ref{thm-Normal-Mapping_Sol_Existence0} and Lemma~\ref{lem-ful-rank}, we have the following result.
 
\begin{theorem}\label{thm-Normal-Mapping_Sol_Existence}
Let $K\subseteq \mathbb{R}^{m}$ be a nonempty closed convex set, the mapping $F:K\to\mathbb{R}^m$ be locally Lipschitz continuous on $K$, and the normal mapping $F^{\rm nor}_K(\cdot)$ be norm coercive. Assume 
that the generalized Jacobian $\partial F^{\rm nor}_K(x)$ is of maximal rank at every point $x\in\mathbb{R}^m$ where $F_K^{nor}(x)\neq 0$.
Then, the set SOL$(K,F)$ is non-empty and compact.
\end{theorem}

When $F$ is continuously differentiable on $K$,} Theorem~\ref{thm-Normal-Mapping_Sol_Existence} is an alternative to Corollary~2 of \cite{NonMonotoneVI}, where we have used the natural map associated with VI$(K,F)$ to establish a sufficient condition for the existence of solutions.

\begin{remark}\label{rem-generalization}
When $K=\mathbb{R}^m$, the normal mapping $F^{\rm nor}_K(\cdot)$ coincides with the mapping $F(\cdot)$. Therefore, when $F(\cdot)$ is differentiable, we have that $\partial F^{\rm nor}_K(x)=\{\nabla F(x)\}$ at every point $x$. Thus, when $K=\mathbb{R}^m$, the conditions of 
Theorem~\ref{thm-Normal-Mapping_Sol_Existence}  are the same as the conditions of 
Theorem~\ref{Thm-main theorem-prime}, which require the mapping $F(\cdot)$ be norm coercive and 
    $\det\left(\nabla F(x)\right)\neq 0$
    for every $x\in \mathbb{R}^{m}$ where $F(x)\neq 0$. 
    Hence, Theorem~\ref{thm-Normal-Mapping_Sol_Existence} is a generalization of Theorem~\ref{Thm-main theorem-prime} to the case of a closed convex set $K\subset\mathbb{R}^m$.
    \hfill$\square$
\end{remark}

Theorem~\ref{thm-Normal-Mapping_Sol_Existence}
 provides a sufficient condition for the existence of a solution to VI$(K,F)$ assuming that 
 $F^{\rm nor}_K(\cdot)$ is norm coercive and $\partial F^{\rm nor}_K(x)$ has maximal rank at every $x$ where $F^{\rm nor}_K(x)\ne 0$. 
 In the sequel, we examine when these two  conditions are satisfied. 

 \subsection{Norm Coercivity of Normal Mapping}\label{ssec-normcoercivity}
 Here, we consider the norm coercivity of the normal mapping.
 It can be seen that, when $K$ is a compact convex set and $F(\cdot)$ is continuous on $K$, then the normal mapping $F_K^{\rm nor}(\cdot)$ is norm coercive. Next, we show that $F_K^{\rm nor}(\cdot)$ is norm coercive when $K$ has Cartesian structure, $K=K_1\times\cdots\times K_N$, and $F(\cdot)$ is a ``uniform block P-function on $K$" in the sense of ~\cite[Definition 2]{parise2019variational}\footnote{In~\cite{facchinei2003finite}, this corresponds to the notion of a uniform P-function on $K$ according to Definition 3.5.8(d).}, 
 requiring  the existence of a scalar $\mu>0$ such that for all $x,y\in K_1\times\cdots\times K_N,$
 \begin{equation}\label{eq-unip-fp}\max_{1\le j\le N}\langle[F(x)-F(y)]_j, [x-y]_j\rangle\geq \mu\|x-y\|_2^2,\end{equation}
where $[\cdot]_j$ denotes the $j$-th block of the vector.
This property is more restrictive than the uniform P-function of Definition~\ref{Def-P-function}, which applies to any set $K$.

\begin{lemma}\label{Lem-coercivity of F-K-Nor}
    Let $K=K_1\times\cdots\times K_N\subseteq\mathbb{R}^m$ with each $K_j$ being a closed convex set, and let $F: K\to \mathbb{R}^m$ satisfy the relation in~\eqref{eq-unip-fp}. 
    Also, assume that $F(\cdot)$ satisfies the following condition 
    for some $p\ge1$, $L_0\ge0$, and $L_p>0$, 
    \[\|F(x)-F(y)\|_2\le L_0+L_p\|x-y\|_2^p\qquad\hbox{for all }x,y\in K.\]
    Then, the normal mapping $F_K^{\rm nor}(\cdot)$ is norm coercive on $\mathbb{R}^m$. 
\end{lemma}
\begin{proof}
By the relation in~\eqref{eq-unip-fp}, there is a $\mu>0$ such that for all $x,y\in \mathbb{R}^m$,
    \begin{equation}\label{eq-F-Nor-1-coer}
        \max_{1\le j\le N}\left\langle [F(\Pi_K[x])-F(\Pi_K[y])]_j, [\Pi_K[x]-\Pi_K[y]]_j\right\rangle\geq \mu\|\Pi_K[x]-\Pi_K[y]\|_2^2.
    \end{equation}
    Let $i$ be the index maximizing the expression on the left-hand side of~\eqref{eq-F-Nor-1-coer}.  
    Note that $[\Pi_K[x]]_{i}=\Pi_{K_i}[x_i]$, where $x_i$ is the $i$th block of $x$ and $K_i\subseteq\mathbb{R}^{m_i}$. For the projection on a closed convex set $K_i\subseteq\mathbb{R}^{m_i}$ we have (Lemma~12.1.13 a) \cite{facchinei2003finite})
    for all $x_i,y_i\in \mathbb{R}^{m_i}$,
    \[\left\|x_i-y_i-\left(\Pi_{K_i}[x_i]-\Pi_{K_i}[y_i]\right)\right\|^2_2
    +\left\|\Pi_{K_i}[x_i]-\Pi_{K_i}[y_i]\right\|^2_2\leq \|x_i-y_i\|_2^2.\]
    By expanding the first element of the preceding inequality, we obtain for all $x_i,y_i\in \mathbb{R}^{m_i}$,
    \[\left\langle x_i-y_i, \Pi_{K_i}[x_i]-\Pi_{K_i}[y_i]\right\rangle - \left\|\Pi_{K_i}[x_i]-\Pi_{K_i}[y_i]\right\|_2^2\geq 0.\]
    By combining this relation with~\eqref{eq-F-Nor-1-coer}, and recalling that the maximum in~\eqref{eq-F-Nor-1-coer} is attained on the index $i$, we have for all $x,y\in \mathbb{R}^m$,
    \[\max_{1\le j\le N}
    \left\langle [F_K^{\rm nor}(x)-F_K^{\rm nor}(y)]_j, [\Pi_K[x]-\Pi_K[y]]_j\right\rangle \geq \mu\|\Pi_K[x]-\Pi_K[y]\|_2^2.\]
    By using H\"older's inequality, we find that for all $x,y\in\mathbb{R}^{m}$,
\begin{align*}
    \|F_K^{\rm nor}(x) -F_K^{\rm nor}(y)\|_2 &\|\Pi_K[x]-\Pi_K[y]\|_2 \cr
    &\ge  \sum_{j=1}^{N}\left|
[F_K^{\rm nor}(x)-F_K^{\rm nor}(y)]_j[\Pi_K[x]-\Pi_K[y]]_j\right|\cr
&\ge \max_{1\le j\le N}
    \left\langle [F_K^{\rm nor}(x)-F_K^{\rm nor}(y)]_j, [\Pi_K[x]-\Pi_K[y]]_j\right\rangle.
\end{align*}
By combining the preceding two relations, we obtain
for all $x,y\in\mathbb{R}^{m}$, 
\begin{align}\label{eq-norm-coercive-1}
    \|F_K^{\rm nor}(x)-F_K^{\rm nor}(y)\|_2\geq \mu \|\Pi_K[x]-\Pi_K[y]\|_2.
\end{align}

To arrive at a contradiction, assume that 
 $F_K^{\rm nor}(\cdot)$ is not norm coercive. Then, there exists an unbounded sequence $\{x^k\}$ with 
 $\lim_{k\to\infty}\|x^k\|_2=+\infty$ such that the norm sequence $\{\|F_K^{\rm nor}(x^k)\|_2\}$ is bounded,
 i.e., for some $\eta>0$,
 \begin{equation}\label{eq-fbound}
     \|F_K^{\rm nor}(x^k)\|_2\le \eta\qquad\hbox{for all } k.
 \end{equation}
 Then, for an arbitrary but fixed vector $v\in K$, we have
 for all $k$,
 \begin{align}\label{eq-eqbound}
     \|F_K^{\rm nor}(x^k)- F_K^{\rm nor}(v)\|_2 &\le \|F_K^{\rm nor}(x^k)\|_2 +\|F_K^{\rm nor}(v)\|_2\cr
     &\le \eta+\|F_K^{\rm nor}(v)\|_2.
 \end{align}
 Since $F_K^{\rm nor}(x)=x-\Pi_K\left[x\right]+F(\Pi_K\left[x\right])$ by the triangle inequality, it follows that for all $k$,
 \begin{align*}
\|x^k-v\|_2 \le & \|F_K^{\rm nor}(x^k)- F_K^{\rm nor}(v)\|_2 \cr
& +\left\|\Pi_K\left[x^k\right]-\Pi_K\left[v\right] -\left(F(\Pi_K\left[x^k\right]) - F(\Pi_K\left[v\right]\right)\right\|_2\cr
\le & \eta+\|F_K^{\rm nor}(v)\|_2 \cr
&+\left\|\Pi_K\left[x^k\right]-\Pi_K\left[v\right] -\left(F(\Pi_K\left[x^k\right]) - F(\Pi_K\left[v\right]\right)\right\|_2,
 \end{align*}
 where the last inequality is obtained by using~\eqref{eq-eqbound}. By using the triangle inequality again, from the preceding relation we have 
 for all $k$,
 \begin{align*}
\|x^k-v\|_2 \le & \eta+\|F_K^{\rm nor}(v)\|_2 +\left\|\Pi_K\left[x^k\right]-\Pi_K\left[v\right]\right\|_2 +\left\|F(\Pi_K\left[x^k\right]) - F(\Pi_K\left[v\right]\right\|_2\cr
\le & \eta+ L_0+\|F_K^{\rm nor}(v)\|_2 +\left\|\Pi_K\left[x^k\right]\!-\!\Pi_K\left[v\right]\right\|_2 +L_p\!\left\|\Pi_K\left[x^k\right] \!-\!\Pi_K\left[v\right]\right\|_2^p\cr
\le &\eta+ L_0+\|F_K^{\rm nor}(v)\|_2 \cr
&+(1+L_p)\left(\left\|\Pi_K\left[x^k\right]\!-\!\Pi_K\left[v\right]\right\|_2 
+\left\|\Pi_K\left[x^k\right] \!-\!\Pi_K\left[v\right]\right\|_2^p\right),
 \end{align*}
 where the second inequality is obtained using the $L_0-L_p$ assumption on the mapping $F(\cdot)$.
 Since the vector $v\in K$ is fixed, and $\|x^k\|_2\to+\infty$, from the preceding relation we conclude that 
\begin{equation*}%\label{eq-limvalue}
    \lim_{k\to\infty}\left\|\Pi_K\left[x^k\right]-\Pi_K\left[v\right]\right\|_2 =+\infty.
\end{equation*}
The preceding relation and relation~\eqref{eq-norm-coercive-1}, with $x=x^k$ and $y=v$, yield 
\begin{equation*}%\label{eq-norm-coercive-2}
    \lim_{k\to\infty}\|F_K^{\rm nor}(x^k)-F_K^{\rm nor}(v)\|_2=+\infty.
\end{equation*}
Since $v\in K$ is a fixed vector, we conclude that 
$\lim_{k\to\infty}\|F_K^{\rm nor}(x^k)\|_2=+\infty$,
which contradicts our assumption made in~\eqref{eq-fbound} that the sequence $\{\|F_K^{\rm nor}(x^k)\|_2\}$ is bounded. Hence, we must have that 
$F_K^{\rm nor}(\cdot)$ is norm coercive.
\hfill $\square$ 
\end{proof}

The $L_0-L_p$ assumption on the mapping $F(\cdot)$ of Lemma~\ref{Lem-coercivity of F-K-Nor} is satisfied, for example, when $F(\cdot)$ is Lipschitz continuous on the set $K$ (corresponding to $L_0=0$ and $p=1$).

\begin{remark}
Lemma~4 in~\cite{2510.02724} provides an alternative condition for the norm coercivity of $F_K^{\rm nor}(\cdot)$ given that the mapping $F(\cdot)$ is norm coercive over the set $K$, where $K$ need not have a Cartesian structure.
%the normal mapping $F_K^{\rm nor}(\cdot)$ is either norm coercive, or there exists a sequence $\{x^k\}\subset \mathbb{R}^m$ with $\lim_{k\to \infty}\|x^k\|_2^2=\infty$ such that
%\begin{equation} \label{eq-Assumption-cor}
%\lim_{k\to \infty}\frac{\langle x^k-\Pi_K[x^k],F(\Pi_K[x^k]) \rangle}%{\|x^k-\Pi_K[x^k]\|_2\|F(\Pi_K[x^k])\|_2}=-1.
%\end{equation}
%Thus, when the later case does not occur $F_K^{\rm nor}(\cdot)$ is norm coercive
%In this case, we use a proof by contradiction. Let assume $\|F_K^{\rm nor}(\cdot)\|_2$ is bounded, the proof of this lemma demonstrates that if $\|\Pi_K(x^k)\|_2$ is bounded, then $\|F_K^{\rm nor}(\cdot)\|_2$ must be unbounded, which leads to a contradiction. Therefore, for sequence $\{x^k\}$ and $y=y_0\in \mathbb{R}^m$, we have by \eqref{eq-norm-coercive-1} that $\|F_K^{\rm nor}(x^k)\|_2\to \infty$ as $\{\|\Pi_K[x^k]\|_2\}\to \infty$. This can be alternative way to arrive a contradiction in the proof of Lemma~\ref{Lem-coercivity of F-K-Nor}.
\end{remark}

% Lemma 3.1 requires Cartesian structure on $K$ -so the remeark below does not work anymore
%\begin{remark}  If $\nabla F(\cdot)$ is a uniform P-matrix on the closed convex set $K$, then $F(\cdot)$ will be a uniform P-function on this set by \cite[Theorem~D.1]{parise2019variational}. Hence, $F_K^{\rm nor}(\cdot)$ is norm coercive by Lemma~\ref{Lem-coercivity of F-K-Nor}.
%\end{remark}
 
 \subsection{Maximal Rank Condition for Generalized Jacobian of Normal Mapping}\label{ssec-gen-jac-normal-map}
 Here, we consider the conditions ensuring maximal rank of the generalized Jacobian $\partial F^{\rm nor}_K(x)$ of the normal mapping $F^{\rm nor}_K(x)$,
 \[F_K^{\rm nor}(x)=x-\Pi_K\left[x\right]+F(\Pi_K\left[x\right])
\qquad\hbox{for all $x\in\mathbb{R}^{m}$}.\]

For the generalized Jacobian $\partial F^{\rm nor}_K(x)$, we will make use of some calculus rules given in the following lemma, where we use $\text{Conv}(X)$ to denote the convex hull of a set $X$. We also 
use the Minkowski sum of two sets $X\subseteq \mathbb{R}^m$ and $Y\subseteq \mathbb{R}^m$, defined by
$X+Y=\{x+y\mid x\in X,y\in Y\}$.

\begin{lemma}
    \label{Lem-properties-Gen-Jacob} For generalized Jacobian, the following basic operations hold:
    \begin{itemize}
        \item [(a)] If $F,G:\mathbb{R}^{m}\to\mathbb{R}^m$ are locally Lipschitz at $x\in \mathbb{R}^m$, then $\partial (F+G)(x)\subseteq \partial F(x)+\partial G(x)$. Moreover, 
         if $G:\mathbb{R}^{m}\to\mathbb{R}^m$ is differentiable at $x\in \mathbb{R}^m$, then $\partial (F+G)(x)=\partial F(x)+\nabla G(x)$.
        \item [(b)] If $F:\mathbb{R}^{m}\to\mathbb{R}^m$ is locally Lipschitz at $x$ and $G:\mathbb{R}^{m}\to\mathbb{R}^m$ is locally Lipschitz at $F(x)$, then  $\partial (G(F))(x)\subseteq\text{Conv}\{\partial G\left(F(x)\right)\partial F(x)\}$. Additionally,
        if $G:\mathbb{R}^{m}\to\mathbb{R}^m$ is differentiable at $F(x)$, then  $\partial (G(F))(x)=\nabla G(F(x))\partial F(x)$,
    \end{itemize}
    where the product $\mathcal{M}_1\mathcal{M}_2$ of two sets $\mathcal{M}_1$ and $\mathcal{M}_2$ of $m\times m$ matrices is given by
    $\mathcal{M}_1\mathcal{M}_2 =\{M_1M_2\mid M_1\in\mathcal{M}_1, \ M_2\in\mathcal{M}_2\}$.
\end{lemma}
\begin{proof}
    For the statements in part (a) see \cite[ Corollary 3.5]{pales2007infinite} and \cite[ Corollary 4.5]{pales2008infinite}, respectively.
    For the statements in part (b) see \cite[Theorem~2.6.6]{clarke1990optimization} (or~\cite[Theorem~4]{imbert2002support}).
    \hfill $\square$
\end{proof}

As seen from the definition of the normal mapping $F_K^{\rm nor}(\cdot)$, assuming that $F(\cdot)$ is differentiable on $K$,
for the generalized Jacobian $\partial F_K^{\rm nor}(\cdot)$, we have by Lemma~\ref{Lem-properties-Gen-Jacob} for all $x\in\mathbb{R}^m$,
\[\partial F_K^{\rm nor}(x)\subseteq
\mathbb{I} -\partial \Pi_K\left[x\right]+\nabla F(\Pi_K\left[x\right]) \partial \Pi_K\left[x\right],\]
where $\partial \Pi_K\left[x\right]$ is not readily available at all points $x$ in general. 

We have the following lemma that gives some basic properties of the generalized Jacobian 
$\partial \Pi_K\left[\cdot\right]$ for the projection mapping.

\begin{lemma}\label{lemma-genjac-proj}
Let $K\subseteq\mathbb{R}^m$ \textcolor{black}{be a nonempty} closed convex set. 
Then, the generalized Jacobian 
$\partial \Pi_K\left[x\right]$ is nonempty compact convex set for all $x\in\mathbb{R}^m$ and we have:
\begin{itemize}
\item[(a)]
For all $x$ in the interior $K^\circ$ of the set $K$, the projection mapping
$\Pi_K\left[x\right]$ is differentiable and 
$\nabla \Pi_K\left[x\right]=\mathbb{I}$ for all $x\in K^\circ$.
\item[(b)] 
If $K^\circ\ne\emptyset$, then
$\mathbb{I}\in\partial \Pi_K\left[x\right]$ for all $x$ on the boundary {\rm bd}$K$ of the set~$K$.
%\item[(c)] 
%At all points $x$ outside the set $K$, we have $I\notin\partial \Pi_K\left[x\right]$.
\end{itemize}
\end{lemma}
\begin{proof}
The projection mapping is Lipschitz continuous with the constant $L=1$. Thus, by Theorem~\ref{Thm-Proposition 1 of Clark paper}, the generalized Jacobian 
$\partial \Pi_K\left[x\right]$ is nonempty compact convex set for all $x\in\mathbb{R}^m$.

For part (a), we note that for any $x\in K^\circ$, there is a neighborhood of $x$ contained in the interior $K^\circ$ and such that 
$\Pi_K\left[v\right]= v$ for all $v$ in the neighborhood of $x$.
Hence, the mapping $\Pi_K\left[\cdot\right]$ is differentiable at $x$ and $\nabla \Pi_K\left[x\right]=\mathbb{I}$. 

For part (b), we note that a point $x\in{\rm bd}K$ can be approached with a sequence $\{x^k\}\subset K^\circ$ if the set $K^\circ$ is nonempty. By part (a), $\Pi_K\left[\cdot\right]$ is differentiable at $x^k\in K^\circ$ for all $k$ and $\nabla \Pi_K\left[x^k\right]=\mathbb{I}$ for all $k$. Thus, by Definition~\ref{Def-Generalized Jacobian}, it follows that $\mathbb{I}\in\partial \Pi_K\left[x\right]$.
 \hfill\hfill $\square$
\end{proof}

As seen from Lemma~\ref{lemma-genjac-proj}, not much can be said about the structure of the generalized Jacobian
$\partial \Pi_K[x]$ when $x\not\in K^\circ$, which makes Lemma~\ref{lemma-genjac-proj} uninformative when the convex set $K$ has empty interior. To deal with this,
in the following development, we will assume that, for $x\not\in K^\circ$, the generalized Jacobian
$\partial \Pi_K[x]$ is contained in the convex hull of the set $G=\left\{\mathbb{I}-e_ie_i^\top\mid i=1,\ldots,m\right\}\cup \{\mathbb{I}\}$ of $m\times m$ matrices, 
where $\{e_1,\ldots, e_m\}$ is the standard Euclidean basis. 
Later we will discuss when this assumption is valid.

Next, we state our main result showing that 
$\partial (tF)_K^{\rm nor}(\cdot)$, with $t>0$, is of maximal rank on $\mathbb{R}^{m}$, which can be used in Theorem~\ref{thm-Normal-Mapping_Sol_Existence}, combined with Theorem~\ref{Thm-auxiliary-normal map solution}, to establish the existence of solutions to VI$(K,tF)$, with $t>0$. By noting that VI$(K,F)$ and VI$(K,\tau F)$ have the same solution set for any $\tau>0$, the result is applicable to investigation of solutions to 
VI$(K,F)$.
\begin{theorem}\label{thm-general jac of Const}
    Let $K\subseteq \mathbb{R}^{m}$ be a closed convex set and $F:K\to \mathbb{R}^{m}$ be a continuously differentiable mapping. 
    Assume that the Jacobian $\nabla F(\cdot)$ is full rank on the set $K$ 
    and the smallest singular value $\sigma_{min}\left(\nabla F(x)\right)$ is uniformly lower bounded by a positive scalar for all boundary points $x$ of $K$. 
    Also, assume that all $(m-1)\times (m-1)$ principal minors of $\nabla F(x)$ are non-zero for every boundary point $x$ of $K$. 
    Moreover, assume that for all $x$ that do not belong to the interior $K^\circ$ of the set $K$, 
    $\partial \Pi_K[x]\subseteq \text{Conv}(G)$, where $G=\left\{\mathbb{I}-e_ie_i^\top\mid i=1,\ldots,m\right\}\cup \{\mathbb{I}\}$.
    Then, there is $t>0$ such that $\partial (tF)_K^{\rm nor}(x)$ is of maximal rank for all $x\in\mathbb{R}^{m}$.
\end{theorem}
\begin{proof}
    Let $t>0$ be arbitrary, and consider the mapping $tF$.
By our assumption, the mapping $tF$ is continuously differentiable on the set $K$, while by Lemma~\ref{lemma-genjac-proj}, $\Pi_K\left[\cdot\right]$ is differentiable at $x\in K^\circ$, so we have
\[\partial (tF)_K^{\rm nor}(x) = \{t\nabla F(\Pi_K\left[x\right])\}\qquad\hbox{for all }x\in K^\circ.\]
By our assumption, the Jacobian $\nabla F(\cdot)$ is full rank on the set $K$ and, therefore, $\partial (tF)_K^{\rm nor}(x)$ also has a full rank for $x\in K^\circ$ and any $t>0$.

Next, we consider $x\not\in K^\circ$, in which case $\Pi_K[x]$ lies on the boundary of the set $K$.
Let us distinguish between the identity operator $I(\cdot)$ and the identity matrix $\mathbb{I}$. 
Thus, we can write
\[(tF)_K^{\rm nor}(x)=I(x)+(tF-I)(\Pi_K[x]).\]
Note that $(tF-I)(\cdot)$ is a continuously differentiable mapping, whose Jacobian is given by $\nabla (tF-I)(x) = t\nabla F(x)-\mathbb{I}$ for all $x$, where we use the fact that
the Jacobian of the identity operator is equal to the identity matrix $\mathbb{I}$ at any point $x$. Hence, by Lemma~\ref{Lem-properties-Gen-Jacob} we have
\[\partial (tF)_K^{\rm nor}(x)
 \subseteq
\mathbb{I} +\left(t\nabla F(\Pi_K\left[x\right]) -\mathbb{I}\right)\partial \Pi_K[x].\]
By our assumption, for $x\not\in K^\circ$, the generalized Jacobian 
$\partial \Pi_K\left[x\right]$ is contained in the convex hull of the matrix set $G$, so we have
\begin{equation}\label{eq-tfnor}
    \partial (tF)_K^{\rm nor}(x)
 \subseteq
\mathbb{I} +\left(t\nabla F(\Pi_K\left[x\right]) -\mathbb{I}\right){\rm Conv}(G),
\end{equation}
where $G=\left\{\mathbb{I}-e_ie_i^\top\mid i=1,\ldots,m\right\}\cup \{\mathbb{I}\}$.
Any element $M$ in the 
convex hull of $G$, can be written as
\begin{equation}\label{eq-convhullg}
    M=(1-\beta)I +\beta Q, \quad\hbox{for some } \beta\in[0,1],
\end{equation}
\[Q=\mathbb{I}-\sum_{i=1}^m \alpha_ie_ie_i^\top
\qquad \sum_{i=1}^m\alpha_i=1, \  \alpha_i\ge0 \hbox{ for all } i=1,\ldots,m.\]
Combining this with equation~\eqref{eq-convhullg}, we have
\begin{align*}
{\rm Conv}(G)=\left\{\mathbb{I}-\beta \sum_{i=1}^{m}\alpha_i e_i e_i^\top\mid \beta\in[0,1],\sum_{i=1}^m\alpha_i=1,\alpha_i\geq 0 \hbox{ for all } i=1,\ldots,m\right\}.\end{align*}
We now use the preceding characterization of ${\rm Conv}(G)$
in order to simplify the right-hand side of the inclusion
for $\partial (tF)_K^{\rm nor}(x)$ in relation~\eqref{eq-tfnor}. For any matrix $\mathbb{I}-\beta \sum_{i=1}^{m}\alpha_ie_ie_i^\top\in {\rm Conv}(G)$, by simple matrix algebra, we have
\begin{align*}
    \mathbb{I} +\left(t\nabla F(\Pi_K\left[x\right]) -\mathbb{I}\right)\left(\mathbb{I}-\beta\sum_{i=1}^{m}\alpha_ie_ie_i^\top\right)
    = &
\beta\sum_{i=1}^{m}\alpha_ie_ie_i^\top \cr
&+ t\nabla F(\Pi_K\left[x\right])\left(\mathbb{I}-\beta\sum_{i=1}^{m}\alpha_ie_ie_i^\top\right).
\end{align*}
Thus, from relation~\eqref{eq-tfnor}, we see that
 $\partial (tF)_K^{\rm nor}(x)$ is contained in the set of matrices that have the form 
 \begin{equation}\label{eq-main-comp}
 \beta\sum_{i=1}^{m}\alpha_ie_ie_i^\top 
+ t\nabla F(\Pi_K\left[x\right])\left(\mathbb{I}-\beta\sum_{i=1}^{m}\alpha_ie_ie_i^\top\right),
\end{equation}
where $\beta\in[0,1]$, $\sum_{i=1}^m\alpha_i=1,$ with $\alpha_i\geq 0$ for all $i$.

It remains to show that
the matrices of the form~\eqref{eq-main-comp} are of the full rank. If $\beta=0$, then the matrix of the form~\eqref{eq-main-comp} reduces to $t\nabla F(\Pi_K\left[x\right])$, which has a full rank for any $t>0$ 
due to our assumption that the Jacobian $\nabla F(\cdot)$ has a full rank on the set $K$.

We now consider separately two cases, the case when $\alpha_i<1$
for all $i$ and $\beta\in(0,1]$, or $\alpha_i=1$ for some $i$ and $\beta\in(0,1)$, and the case when $\alpha_i=1$ for some $i$ and $\beta=1$.
These cases correspond to having the matrix 
$\mathbb{I}-\beta\sum_{i=1}^{m}\alpha_ie_ie_i^\top$ in~\eqref{eq-main-comp} invertible or not.

{\it Case $\alpha_i<1$ for all $i$ and $\beta\in(0,1]$, or $\alpha_i=1$ for some $i$ and $\beta\in(0,1)$.} 
Recalling that the vectors $e_1,\ldots,e_m$ are orthonormal, we have that the eigenvalues of 
$\sum_{i=1}^{m}\alpha_ie_ie_i^\top$ are $\alpha_1,\ldots,\alpha_m$. Hence, the eigenvalues of 
$\mathbb{I}-\beta\sum_{i=1}^{m}\alpha_ie_ie_i^\top$ are
$1-\beta\alpha_i$ for $i=1,\ldots,m$.
Note that $1-\beta\alpha_i>0$ for all $i$ when
$\alpha_i\in[0,1)$ for all $i$ and $\beta\in(0,1]$. Thus, the matrix $\mathbb{I}-\beta\sum_{i=1}^{m}\alpha_ie_ie_i^\top$ in~\eqref{eq-main-comp} is invertible.
When  $\alpha_i=1$ for some $i$, then we must have $\alpha_j=0$ for all $j\ne i$, since $\sum_{i=1}^m\alpha_i=1$. Thus, for $\beta\in(0,1)$, we have 
$1-\beta\alpha_i>0$ and $1-\beta\alpha_j=1$ for all $j\ne i$. Hence, again, the matrix $\mathbb{I}-\beta\sum_{i=1}^{m}\alpha_ie_ie_i^\top$ in~\eqref{eq-main-comp} is invertible.

Scaling the matrix in~\eqref{eq-main-comp}  by $t$ and letting $t\to\infty$, in the limit we obtain the matrix 
 \begin{equation*}%\label{eq-main-comp2}
 \nabla F(\Pi_K\left[x\right])\left(\mathbb{I}-\beta\sum_{i=1}^{m}\alpha_ie_ie_i^\top\right).
\end{equation*}
This matrix is invertible since $\mathbb{I}-\beta\sum_{i=1}^{m}\alpha_ie_ie_i^\top$ is invertible and the Jacobian $\nabla F(\cdot)$ has a full rank on the set $K$ (by our assumption). By continuity of the determinant and $\sigma_{\min}\left(\nabla F(\Pi_K[x])\right)$ being bounded away from zero, we conclude that for sufficiently large $t$  and for every $x\notin K^\circ$, the determinant of the matrix in~\eqref{eq-main-comp}  is non-zero; hence, the matrix is of full rank. Therefore, 
the generalized Jacobian $\partial (tF)_K^{\rm nor}(x)$ is of maximal rank.

{\it Case $\alpha_{i}=1$ for some $i\in [m]$ and $\beta=1$.} 
In this case, $\alpha_j=0$ for all $j\ne i$ and the matrix in~\eqref{eq-main-comp} takes on the following form
\begin{equation}\label{eq-main-comp-2}
    e_{i}e_{i}^\top 
+ t\nabla F(\Pi_K\left[x\right])\left(\mathbb{I}-e_{i}e_{i}^\top\right).
\end{equation}
The ${i}$-th column of the matrix in~\eqref{eq-main-comp-2} is given by the vector $e_i$, and the matrix has the following form:
\[e_{i}e_{i}^\top 
+ t\nabla F(\Pi_K\left[x\right])\left(\mathbb{I}-e_{i}e_{i}^\top\right)
=\begin{bmatrix}
[M_t(x)]_{11} & \;{\bf 0}\;      & [M_t(x)]_{13} \\
[M_t(x)]_{21} & \;1\;      & [M_t(x)]_{23} \\
[M_t(x)]_{31} & \;\bf{0}\;      & [M_t(x)]_{33}
\end{bmatrix},
\]
where, $M_t(x)=t\nabla F(\Pi_K[x])$, and $[M_t(x)]_{pq}$ denotes $pq$-block  decomposition of the matrix $M_t(x)$, and ${\bf 0}$ is the zero vector of an appropriate dimension.
The determinant of the matrix in the preceding relation can be computed by expanding the determinant with respect to the $i$th column. 
In this case, 
\[\det\left(e_{i}e_{i}^\top 
+ t\nabla F(\Pi_K\left[x\right])\left(\mathbb{I}-e_{i}e_{i}^\top\right)\right)=1\cdot \det\left(\begin{bmatrix}
[M_t(x)]_{11}  & [M_t(x)]_{13} \\
[M_t(x)]_{31}  & [M_t(x)]_{33}
\end{bmatrix}\right).\]
By our assumption, all $(m-1)\times (m-1)$ principal minors of $\nabla F(\Pi_K[x])$ are non-zero for all $x\notin K^\circ$, thus implying that the determinant on the right hand side of the preceding relation is non-zero. 
\hfill $\square$
\end{proof}

The condition 
$\partial \Pi_K[x]\subseteq \text{Conv}(G)$, with $G=\left\{\mathbb{I}-e_ie_i^\top\mid i=1,\ldots,m\right\}\cup \{\mathbb{I}\}$, of Theorem~\ref{thm-general jac of Const} is satisfied, for example,
when $K$ is the  Cartesian product of one-dimensional closed convex  sets, i.e., $K=K_1\times\cdots\times K_m,$ with  $K_i\subseteq\mathbb{R}$ being closed and convex set for all $i$. For this special Cartesian product structure of the set $K$, when the mapping $F$ is continuously differentiable
and satisfies the uniform P-function property, the VI$(K,F)$
has a unique solution, as discussed next.
Figure~\ref{fig:relation-2} depicts the implications of the uniform $P$-function property for the Jacobian $\nabla F(\cdot)$.

%\begin{theorem}\label{thm-Pfunct-uniquesol}
%Let $K=K_1\times\cdots\times K_m\subseteq\mathbb{R}^m$, with  $K_i\subseteq\mathbb{R}$ being a closed convex set for all $i$, and let $F: K\to \mathbb{R}^m$ be a $P$-function. 
%    Also, assume that the mapping $F(\cdot)$ is continuously differentiable over the set $K$ and satisfies the following relation 
%    for some $p\ge1$, $L_0\ge0$, and $L_p>0$, 
%    \[\|F(x)-F(y)\|_2\le L_0+L_p\|x-y\|_2^p\qquad\hbox{for all }x,y\in K.\]
%    Then, the VI$(K,F)$ has a unique solution. 
%\end{theorem}
When the conditions of Lemma~\ref{Lem-coercivity of F-K-Nor} are satisfied,
the normal mapping $F_K^{\rm nor}(\cdot)$ is norm coercive. 
When $F(\cdot)$ is a uniform P-function on $K$, by Theorem~\ref{thm-P-function},
the Jacobian $\nabla F(x)$ has a full rank for all $x\in K$ and the smallest singular values $\sigma_{\min}(\nabla F(\cdot))$ of the Jacobian are bounded away from zero uniformly on the boundary of the set $K$. Also, by Theorem~\ref{thm-P-function}, all $(m-1)\times(m-1)$ principal minors are nonzero on the boundary points of $K$. 
Thus, all the conditions in Theorem~\ref{thm-general jac of Const} imposed on the Jacobian $\nabla F(\cdot)$ are satisfied.
Thus, if $K=K_1\times\cdots\times K_m$ and $F(\cdot)$ is a uniform P-function on $K$, all the conditions of Theorem~\ref{thm-general jac of Const} are satisfied.
Thus, by Theorem~\ref{thm-general jac of Const}, there is $t>0$ such that $\partial (tF)_K^{\rm nor}(x)$ is of maximal rank for all $x\in\mathbb{R}^{m}$.
Hence, when $F^{\rm nor}_K(\cdot)$ is norm coercive, all other conditions of Theorem~\ref{thm-Normal-Mapping_Sol_Existence} are satisfied, implying that
the VI$(K,tF)$ has a solution for some $t>0$. 
As the solution sets of VI$(K,t F)$ and VI$(K,F)$ coincide, we conclude that VI$(K,F)$ has a solution. The uniqueness of the solution follow by Proposition 3.5.10(a) in~\cite{facchinei2003finite}, which states that VI$(K,F)$ has at most one solution. However,  when $K=K_1\times\cdots\times K_m$ and $F(\cdot)$ is a uniform P-function on $K$, the same result can be obtained directly from Proposition~3.5.10(b) of~\cite{facchinei2003finite}. 

\begin{figure}[h!]
\centering
  \includegraphics[width=0.85\linewidth]{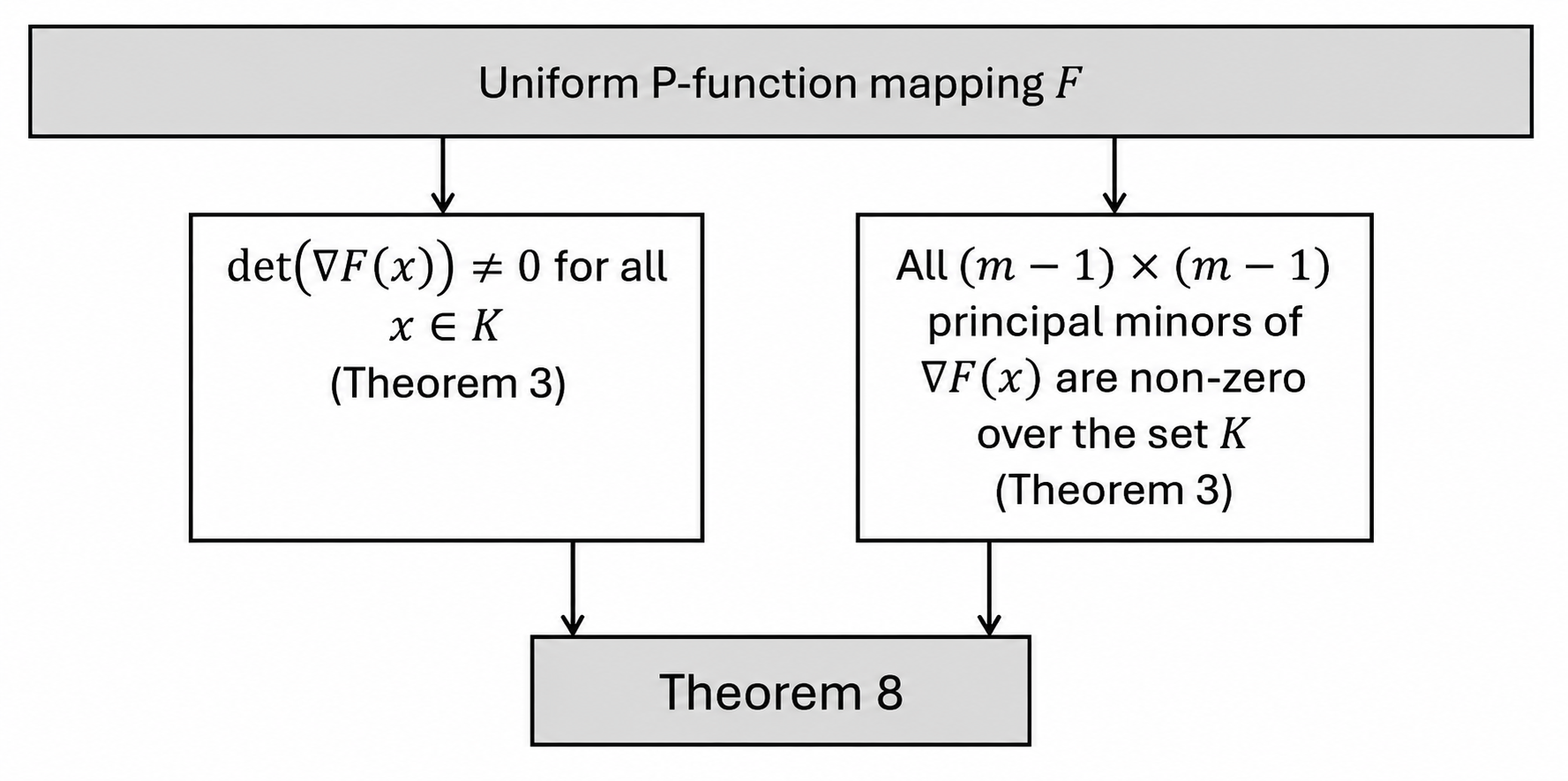}
  \caption{Implications of the uniform P-function property for the conditions on Jacobian used in Theorem~\ref{thm-general jac of Const}.}
  \label{fig:relation-2}
\end{figure}

%--------------------------------------------------------------------------------
\section{Implications for {\color{black}Variational Inequalities} Arising from Games} \label{Sec-Game}
%--------------------------------------------------------------------------------
In this section, we consider a VI arising from a game and analyze
the existence of a (quasi)-Nash equilibrium through the main result of Section~\ref{Sec-Main result-VI}.
Specifically, we consider a game with $N$ players, indexed by $i=1,2,\ldots,N$. The set  $\mathcal{N}=\{1,2,\ldots,N\}$ denotes the collection of all players. The player $i$ decisions are in the set $K_i\subseteq \mathbb{R}^{n_i}$. In a game each player $i$ wants to minimize its cost function $f_i(x_i,x_{-i})$, where $x_i$ is the decision variable of player $i$ and $x_{-i}$ is the concatenation of all other players' decisions $x_j$, i.e., $x_{-i}=(x_1,\ldots,x_{i-1},x_{i+1},\ldots,x_N)$ for all players $i$.
Thus, the player $i$ is confronted with the following minimization problem:
\begin{eqnarray}\label{Eq-game}
    &&\hbox{minimize \ \ \ } f_i(x_i,x_{-i}) \nonumber\\ 
    && \hbox{subject to\ \ \ } x_i \in K_i, 
    %\ \ { x_{-i} \in K_{-i}=\prod_{j=1,j\neq i}^N K_j},
\end{eqnarray}
where $f_i:K \rightarrow \mathbb{R}$ and $K=K_1\times\cdots\times K_N$. The 
cost function $f_i$ is assumed to be known only by player $i$. 
A Nash equilibrium is a solution concept for the non-cooperative game in~\eqref{Eq-game}.
\begin{definition}[Nash Equilibrium]\label{def-Nash}
    A tuple $x^*=(x_1^*,\ldots, x_N^*)$, with $x_i^*\in K_i$ for all $i=1,\ldots,N,$ is a Nash equilibrium of the game in~\eqref{Eq-game} if 
    \[f_i(x_i,x_{-i}^*)\geq f_i(x^*_i,x_{-i}^*)\qquad\hbox{for all }x_i\in K_i, \ \ i=1,2,\ldots,N.\]
\end{definition}

When the game is convex, as defined by Rosen \cite{rosen1965existence}, i.e., the sets $K_i$ are convex and closed, and each $f_i(\cdot,x_{-i})$ is convex and continuously differentiable over $K_i$ for every $x_{-i}\in K_{-i}$, then by the first-order optimality principle, we have that a tuple $x^*=(x_1^*,\ldots, x_N^*)$, with $x_i^*\in K_i$ for all $i$, is a Nash equilibrium for the game in~\eqref{Eq-game} if and only if 
\[\nabla_i f_i(x^*_i,x_{-i}^*)(x_i-x_i^*)\ge 0 \qquad\hbox{for all $x_i\in K_i$ and for all $i=1,2,\ldots,N$},\]
where $\nabla_i f_i(x_i,x_{-i})$ denotes the partial derivative $\nabla_{x_i} f_i(x_i,x_{-i})$. For a convex game, 
Proposition~1.4.2 of~\cite{facchinei2003finite} connects a Nash equilibrium to a solution of a variational inequality problem: a tuple $x^*=(x_1^*,\ldots, x_N^*)$, with $x_i^*\in K_i$ for all $i$, is a Nash equilibrium for the game in~\eqref{Eq-game} if and only if $x^*\in K$ is a solution to VI$(K,F)$, where $K=K_1\times\cdots\times K_N$ and $F(x)$ is the \textcolor{black}{ associated game mapping}, given by $F(x)=[\nabla_1 f_1(x_1,x_{-1}),\ldots,\nabla_N f_N(x_N,x_{-N})]$.

When the sets $K_i$ are convex and closed, and each $f_i(x_i,x_{-i})$ is continuously differentiable in $x_i$ over $K_i$ for every $x_{-i}\in K_{-i}$, but not necessarily convex in $x_i$, 
one can still consider VI$(K,F)$, with $K=K_1\times\cdots\times K_N$ and the \textcolor{black}{ associated game mapping $F(\cdot)$}. 
In this case, VI$(K,F)$ may have a solution, but the solution need not be a Nash equilibrium due to the lack of convexity structure in players' cost functions. To differentiate this case, a quasi-Nash equilibrium concept has been introduced. We use the definition 
that extends \cite[Definition 2.13]{C2} from a two player to a multiple player game.

\begin{definition}[Quasi-Nash Equilibrium]\label{def-quasi-Nash} Let $K_i\subseteq\mathbb{R}^{n_i}$ be a closed convex set for all $i=1,\ldots,N$.
    A tuple $x^*=(x_1^*,\ldots, x_N^*)$,  with $x_i^*\in K_i$ for all $i=1,\ldots,N,$ is a quasi-Nash equilibrium of the game in~\eqref{Eq-game} if 
    \[\langle x-x^*,F(x^*)\rangle\geq 0 \qquad \hbox{for all } x\in K_1\times\cdots\times K_N,\]
    where $F(x)=[\nabla_1 f_1(x_1,x_{-1}),\ldots,\nabla_N f_N(x_N,x_{-N})]$.
\end{definition}
The definition of a quasi-Nash equilibrium
requires that the first-order stationary conditions are satisfied for each player's minimization 
problem in~\eqref{Eq-game}. Thus, a quasi-Nash equilibrium need not be a Nash equilibrium, while a Nash equilibrium is always a quasi-Nash equilibrium.

In what follows, we let $K$ be the Cartesian product of the players' action sets $K_i\subseteq \mathbb{R}^{n_i}$, i.e.,
\[K=K_1\times\cdots\times K_N,\]
and we let $\bar{n}=\sum_{i=1}^N n_i$. The mapping $F:K\to\mathbb{R}^{\bar n}$ has coordinate mappings given by $F_i(x)=\nabla_i f(x_i,x_{-i})$.
In such a VI($K,F)$ arising from a game, the set $K$ is a Cartesian product of the players action sets, in which case properties of the mapping $F(\cdot)$ such as P-function \cite[Proposition~3.5.10]{facchinei2003finite}
or of the Jacobian $\nabla F(\cdot)$ are often used to assert the existence of a Nash equilibrium (such as 
uniform P-matrix and ${\rm P}_\Upsilon$-matrix conditions in~\cite{parise2019variational}). 
As these conditions deal with the Jacobian $\nabla F(\cdot)$, we will assume that the players' cost functions are twice continuously differentiable over the set $K$. Also, we assume that each set $K_i$ is closed and convex. 
%These are assumptions are formalized as follows.

\begin{assumption}\label{Assum-Suff Cond- part1}
For all players $i=1,2,\ldots,N$, the set $K_i$ is closed and convex, and
the function $f_i(x)$ is twice continuously differentiable over the set $K$.
\end{assumption}

In the next two sections, we explore the uniform P-matrix and ${\rm P}_\Upsilon$-matrix conditions, respectively,  and relate them to the results of Section~\ref{Sec-Main result-VI}.
%---------------------------------------------------
\subsection{Uniform P-Matrix Condition for Game Jacobian}\label{Sec-Game-P-Cond}
%---------------------------------------------------
Here, we consider the uniform P-matrix condition on the Jacobian $\nabla F(\cdot)$ associated with the game in~\eqref{Eq-game}. We show that this condition 
implies that each player's cost function $f_i$ is strongly convex with respect to the player's own decision variable, and
that the Jacobian is non-singular over the set $K$. Then, we relate these findings with our results of Section~\ref{Sec-Main result-VI}.

In the sequel, we use $\nabla^2_{ii}f_i(x_i,x_{-i})$ to denote $\nabla_{x_i^2} f_i(x)$ for all $i$.

\begin{theorem}\label{thm:Pmatrix_strongly_convex}
Let Assumption~\ref{Assum-Suff Cond- part1} hold. Assume that the Jacobian $\nabla F(\cdot)$ satisfies the uniform P-matrix condition on the set $K$ (Definition~\ref{Def-P matrix}).  
Then, we have
\begin{itemize}
    \item [(a)]  There is a uniform positive lower bound for the eigenvalues of the matrices
$\nabla^2_{i i} f(x_i,{x}_{-i})$ for all $x_i\in K_i$, all $x_{-i}\in K_{-i}$, and for all $i$.
%where $\nabla^2_{ii}f_i(x_i,x_{-i})=\nabla_{x_i^2} f_i(x)$. 
In particular,
    for every player $i=1,\ldots,N$, the function $f_i(x_i,{x}_{-i})$ is strongly convex in $x_i\in K_i$ for all $x_{-i}\in K_{-i}$.
    \item [(b)] The Jacobian $\nabla F(x)$ has no zero eigenvalue at any point $x\in K$.
    \item [(c)] Every principal minor of  $\nabla F(x)$ is non-zero for all $x\in K$.
\end{itemize}

\end{theorem}
\begin{proof}
For part (a), we let $x\in K$ and $i\in\mathcal{N}$ be arbitrary.
Let $\lambda_{\min}(x)$ be the smallest eigenvalue of \(\nabla^2_{i i}f_i(x)\)
and $v(x)\ne 0$ be its associated eigenvector. 
Define $w(x)\in\mathbb{R}^{\bar n}$
by 
\begin{equation}\label{eq-defw}
w(x)=\bigl(0,\dots,0,\underbrace{v(x)}_{i\text{‑th block}},0,\dots,0\bigr)^\top.
\end{equation}
Since the Jacobian $\nabla F(x)$ satisfies the uniform P-matrix condition on $K$, there exists $\eta>0$ and, for the given $w$ and ${\bf x}_m$ consisting of $m$ copies of the vector $x\in K$, there exists positive definite block‑diagonal matrix
$H_{w(x),x}$
 such that 
\begin{equation}\label{eq:P_quad}
\langle w(x),H_{w(x),x}\,\nabla F(x) w(x)\rangle \ge\eta\| w(x)\|_2^{2}.
\end{equation}
By the definition of the vector $w(x)$ in~\eqref{eq-defw}, it follows that
\begin{align*}
\langle w(x),H_{w(x),x}\,\nabla F(x) w(x)\rangle 
& =\langle v(x),D_i(x)\nabla^2_{i i}f_i(x) v(x)\rangle\cr
&=\lambda_{\min}(x)\langle v(x),D_i(x)v(x)\rangle,
\end{align*}
where $D_i(x)$ denotes the $i$-th block diagonal matrix  of 
$H_{w(x),x}$, and the last equality follows from the fact that $v(x)$ is an eigenvector of $\nabla^2_{i i}f_i(x)$ associated with the eigenvalue $\lambda_{\min}(x)$.
By combining the preceding relation with~\eqref{eq:P_quad}
and noting that $\|w(x)\|_2=\|v(x)\|_2$, we obtain
\[\lambda_{\min}(x)\langle v(x),D_i(x)v(x)\rangle\ge \eta\| v(x)\|_2^{2}.\]
Therefore,
\[
\lambda_{\min}\cdot\lambda_{\max}(D_i) \cdot \| v\|_2^{2}\ge\eta\| v\|_2^{2},\]
where $v(x)\ne0$.
Note that $\lambda_{\max}(D_i(x))\le \lambda_{\max}(H_{w,x})$ and 
by the uniform P-matrix condition, we have that
{\color{black} $\sup_{w}\sup_{z\in K^m}\lambda_{\max}(H_{w,z})<\infty$. Hence,
\[\lambda_{\min}(x)\ge\frac{\eta}{C}>0\qquad\hbox{with }\quad C=\sup_{w}\sup_{z\in K^m}\lambda_{\max}(H_{w,z}).
\]}
Thus, every eigenvalue of \(\nabla^2_{i i}f_i(x)\) is bounded below by the positive
constant \(\eta/C\) implying that 
$f_i(x_i,x_{-i})$ is \(\alpha\)-strongly convex in \(x_i\) over $K_i$ for all $x_{-i}\in K_{-i}$ with the strong convexity constant
\(\alpha=\eta/C\).

The statements in parts (b) and (c) follow from Theorem~\ref {Thm-Singval-Pmat}, since the Jacobian $\nabla F(\cdot)$ satisfies the uniform P-matrix condition over the set $K$. 
\hfill $\square$
\end{proof}

The following result is a consequence of Lemma~\ref{Lem-coercivity of F-K-Nor} and the uniform P-matrix condition on the game Jacobian.
\begin{theorem}\label{thm:Pmatrix_norm_coercive}
      Let Assumption~\ref{Assum-Suff Cond- part1} hold, and  let $\nabla F(\cdot)$ satisfy the uniform P-matrix condition over the set $K$. Also, assume that the mapping $F(\cdot)$  satisfies the following relation
    for some $p\ge1$, $L_0\ge0$, and $L_p>0$, 
    \[\|F(x)-F(y)\|_2\le L_0+L_p\|x-y\|_2^p\qquad\hbox{for all }x,y\in K.\]
    Moreover, assume that each player's set $K_i\subseteq \mathbb{R}^{n_i}$ is a Cartesian product of $n_i$ one-dimensional closed convex sets, i.e., $K_i=K_{i1}\times\cdots \times K_{in_i}$ with each $K_{ij}$ being closed and convex for all $i$ and $j$.
      Then, the normal mapping $F_K^{\rm nor}(\cdot)$ associated with the game~\eqref{Eq-game} is norm coercive.
\end{theorem}
\begin{proof}
    The uniform P-matrix condition implies that the mapping $F(\cdot)$ is a uniform P-function on the set $K$ by Proposition~3.c) of \cite{parise2019variational}. Hence, by the definition of the uniform P-function
    on $K$ (see Definition~\ref{Def-P-function}), there exists $\mu>0$ such that 
\[\max_{1\le j\le \bar n}[F(x)-F(y)]_j\cdot [x-y]_j\geq \mu\|x-y\|_2^2\qquad \hbox{for all }x,y\in K,\]
where $\bar n=\sum_{i=1}^N n_i$ and $[\cdot]_j$ denotes the $j$th element of the vector. By our assumption on the Cartesian structure of the players' sets $K_i$, we have that $K=X_1\times\cdots\times X_{\bar n}$, where each set $X_i\subseteq\mathbb{R}$ is closed and convex.
Hence, the mapping $F(\cdot)$ satisfies relation~\eqref{eq-unip-fp}, with $m=\bar n$, in terms of the coordinates $j=1,\ldots,\bar n$, and the result follows by
Lemma~\ref{Lem-coercivity of F-K-Nor}. 
\hfill $\square$
\end{proof}

When the Jacobian $\nabla F(\cdot)$ of the game mapping satisfies the uniform P-matrix condition (Definition~\ref{Def-P matrix}) over the set $K$, the Jacobian $\nabla F(\cdot)$  has no zero eigenvalue at any point in the set $K$ by Theorem~\ref{thm:Pmatrix_strongly_convex}(b). Moreover, it has all principal minors positive by Theorem~\ref{thm:Pmatrix_strongly_convex}(c).
In addition, when each player's set $K_i$ is a Cartesian product of $n_i$ one-dimensional closed convex sets, then the set $K=K_1\times\cdots\times K_{\bar n}$, $\bar n=\sum_{i=1}^N n_i$, then 
the conditions of Theorem~\ref{thm-general jac of Const} are satisfied.

If the Jacobian $\nabla F(\cdot)$ satisfies the uniform P-matrix condition, then the normal mapping $F_K^{\rm nor}(x)$ is norm coercive by Theorem~\ref{thm:Pmatrix_norm_coercive}. Hence, the mapping $F(\cdot)$ satisfies the conditions of Theorem~\ref{thm-Normal-Mapping_Sol_Existence}, as long as $K=K_1\times\cdots\times K_{Nn}$ (which holds when each player's set $K_i$ is a Cartesian product of $n_i$ one-dimensional sets). Consequently, the existence of a quasi-Nash equilibrium is guaranteed by Theorem~\ref{thm-Normal-Mapping_Sol_Existence}.
Furthermore, by
Theorem~\ref{thm:Pmatrix_strongly_convex}(a), each  player's cost function is strongly convex with respect to the player's own decision variables. Thus, for the game in~\eqref{Eq-game},
the set of quasi-Nash equilibria coincides with the set of Nash equilibria.
To visualize the implications of the uniform P-matrix condition and its connection to our results, we provide a block-diagram  in Figure~\ref{fig:relation}.

\begin{figure}[h!]
\centering
  \includegraphics[width=0.85\linewidth]{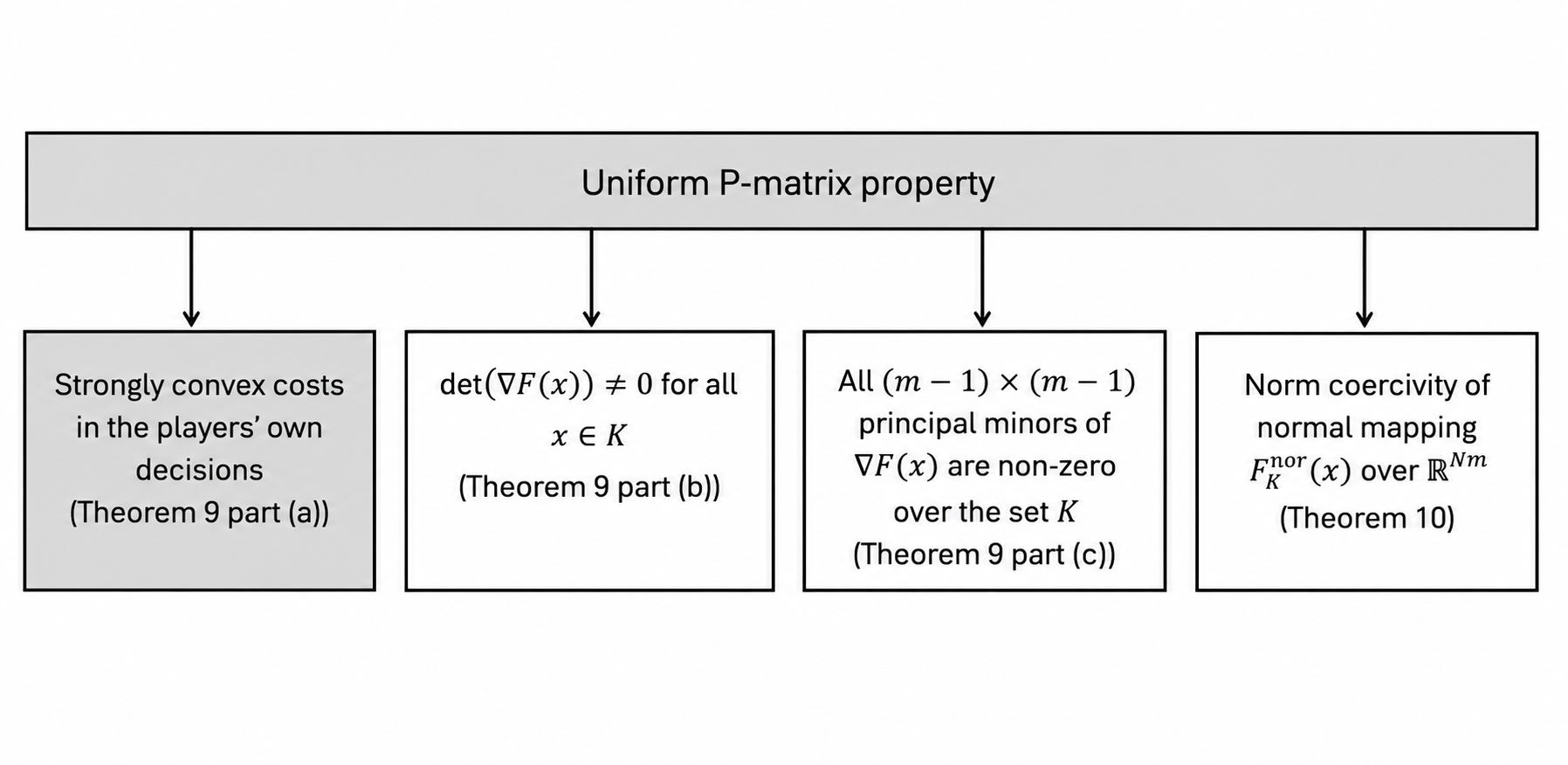}
  \caption{Relation of the uniform P-matrix property to our results.}
  \label{fig:relation}
\end{figure}

In a case of VI$(K,F)$ arising from a game, 
%Definition~\ref{Def-P-function} of a uniform P-function is more general than that of Definition~3.5.8 in~\cite{facchinei2003finite}, requiring that for some $\mu>0$, 
%\[\max_{1\le i\le N}\langle F_i(x)-F_i(y),x_i-y_i\rangle \geq \mu\|x-y\|_2^2 \ \ \forall x,y\in K_1\times\cdots\times K_N,\]
%where the condition is in terms of the players' action sets $K_i$. Thus, we cannot directly apply Proposition~3.5.10(b) of~\cite{facchinei2003finite} to conclude that the VI$(K,F)$ has a (unique) solution.  
%In this case, 
aside from the uniform P-matrix condition, the P$_\Upsilon$-matrix condition~\cite{parise2019variational} can be used to assert the existence of a solution to VI($K,F$), which we 
consider in the next section.

%-------------------------------------------------
\subsection{P$_{\Upsilon}$-Matrix Condition for Game Jacobian}\label{Sec-Game-P_U-Cond}
%-------------------------------------------------
In this section, we consider the game in~\eqref{Eq-game} where each player action set is of the same dimension, i.e., $K_i\subseteq\mathbb{R}^n$ for all $i=1,\ldots,N$. In this case, the game Jacobian $\nabla F(x)$ is an $Nn\times Nn$ matrix and can be decomposed in $N^2$ blocks, each being an $n\times n$ matrix. With such a Jacobian matrix one can associate an $N\times N$ matrix that captures certain properties of the blocks.

To formalize the framework, consider a mapping $x\mapsto A(x)$ defined on the set $K$, where $A(x)$ is an $Nn\times Nn$ matrix with $N^2$ blocks. Let $A_{ij}(x)$ denote the $ij$-th block of $A(x)$. The $\Upsilon$ matrix associated with this mapping  is independent of $x$, and it is defined by
\begin{equation}\label{Upsilon-matrix}
\Upsilon=\begin{bmatrix}
\kappa_{11} &-\kappa_{12} & \ldots & -\kappa_{1N}\\
-\kappa_{21} &\kappa_{22} & \ldots & -\kappa_{2N}\\
 \vdots &\vdots & \ddots & \vdots \\
 -\kappa_{N1} &-\kappa_{N2} & \ldots & \kappa_{NN}\\
\end{bmatrix},
\end{equation}
where 
\[\kappa_{ii}=\inf_{x\in K} \lambda_{\min}(A_{ii}(x))\qquad\hbox{for all $i$},\] 
\[\kappa_{ij}=\sup_{x\in K}\|A_{ij}(x)\|_2\qquad\hbox{for all $i\neq j$}.\]

We have the following definition.
\begin{definition}[${\rm P}_\Upsilon$-Matrix Condition, Definition~5 \cite{parise2019variational}]\label{P-Upsilon}
    Given a set $K$, a matrix-valued mapping $A(\cdot):K\to\mathbb{R}^{Nn\times Nn}$ satisfies ${\rm P}_\Upsilon$-matrix condition over the set $K$ if 
    the diagonal blocks of $A(x)$ are symmetric and positive definite for all $x\in K$, $\kappa_{ij}<\infty$ for all $i,j$, and the matrix $\Upsilon$ in \eqref{Upsilon-matrix} is a P-matrix\footnote{A P-matrix is a matrix that has all principal minors positive, see Definition 3.4 in~\cite{fiedler-ptak-Pmatrices}.}.
\end{definition}

When the game Jacobian $\nabla F(\cdot)$ satisfies the ${\rm P}_\Upsilon$-matrix condition, then the game in~\eqref{Eq-game} has a unique Nash equilibrium.

\begin{theorem}\label{thm-P_Upsilon-uniquesol}
Let Assumption~\ref{Assum-Suff Cond- part1} hold and let $K_i\subseteq\mathbb{R}^n$ for all $i$. Let the Jacobian $\nabla F(\cdot)$ of the game in~\eqref{Eq-game} satisfy the P$_{\Upsilon}$-matrix condition on the set $K$. Then, the game has a unique Nash equilibrium.
\end{theorem}
\begin{proof}
By Proposition 3b) of~\cite{parise2019variational}, if the game Jacobian satisfies the P$_{\Upsilon}$-matrix condition, then the mapping $F(\cdot)$ is a uniform block P-function, i.e., there is $\mu>0$ such that 
 \[\max_{1\le j\le N}\langle[F(x)-F(y)]_j, [x-y]_j\rangle\geq \mu\|x-y\|_2^2
 \qquad\hbox{for all $x,y\in K_1\times\cdots\times K_N$},\]
where $[\cdot]_j$ denotes the $j$-th block of a vector.
By Proposition 2b) of~\cite{parise2019variational} (or Proposition 3.5.10(b) of~\cite{facchinei2003finite}), it follows that the VI$(K,F)$ associated with the game has a unique solution. It remains to show that this solution is the Nash equilibrium of the game.

Since 
the Jacobian $\nabla F(\cdot)$ satisfies the P$_{\Upsilon}$-matrix condition on the set $K$, it follows that all diagonal blocks of the Jacobian are positive definite for all $x\in K$. Moreover, since $\Upsilon$ is a P-matrix, it follows that its diagonal entries are positive, i.e., $\kappa_{ii}=\inf_{x\in K} \lambda_{\min}(A_{ii}(x)) >0$ for all $i$. 
The $i$-th diagonal block of $\nabla F(x)$ is given by $\nabla_{ii}f_i(x)$, so we conclude that, for every $i$, the function $f_i(x)$ is (strongly) convex in $x_i$ over the set $K_i$ for all $x_{-i}\in K_{-i}$, for all $i$. Hence, the solution to the VI$(K,F)$  is the unique Nash equilibrium of the game by Proposition~1.4.2 in~\cite{facchinei2003finite}.
\hfill $\square$
\end{proof}

Theorem~\ref{thm-P_Upsilon-uniquesol} establishes the existence and uniqueness of the Nash equilibrium in game~\eqref{Eq-game} under weaker conditions than that of Theorem~2 in~\cite{parise2019variational}. In particular, in addition to assuming that the players' cost functions $f_i(x_i, x_{-i})$ are twice continuously differentiable in $x$ over the set $K$ and that the game Jacobian satisfies the P$_{\Upsilon}$-matrix condition,
Theorem~2 in~\cite{parise2019variational} requires that the gradients $\nabla_i f_i(x)$ are Lipschitz continuous in $x$ over the set $K$ for all $i$ and that each $f_i(x_i, x_{-i})$ is convex in $x_i$ over the set $K_i$ for all $x_{-i}\in K_{-i}$. In contrast, Theorem~\ref{thm-P_Upsilon-uniquesol} does not require any condition on the gradients $\nabla_i f_i(x)$. Moreover, as seen from the proof of Theorem~\ref{thm-P_Upsilon-uniquesol}, the P$_{\Upsilon}$-matrix condition on the Jacobian $\nabla F(\cdot)$ over the set $K$ implies that each player's cost function  $f_i(x_i, x_{-i})$ is strongly convex in $x_i $  over the set $K_i$ for all $x_{-i}\in K_{-i}$.

Next, we consider a special case of the game in~\eqref{Eq-game} where each player's actions set is the entire space $\mathbb{R}^n$. 
We show that if the VI($\mathbb{R}^{Nn}$) has a quasi-Nash equilibrium 
$\bar x$ and a Polyak–\L{}ojasiewicz (P\L{})-type condition holds for all cost functions at $\bar x$, then $\bar x$ is a Nash equilibrium of the game.

\begin{theorem}\label{Thm-Our+PL}
    Consider the game introduced in \eqref{Eq-game} with $K_i=\mathbb{R}^{n}$ for all $i$, and assume that each 
    $f_i:\mathbb{R}^{Nn}\to \mathbb{R}$ is twice continuously differentiable. Also, assume that  
    the game has a quasi-Nash equilibrium, denoted by $\bar{x}$, 
    satisfying  the following P\L-type condition for all $i$: for $\bar x_{-i}\in \mathbb{R}^{(N-1)n}$ there is $\mu_i(\bar{x}_{-i})>0$ such that for all $x_i\in\mathbb{R}^n$,
\begin{equation}\label{eq-plcond}
        \|\nabla_i f_i(x_i,\bar{x}_{-i})\|_2^2\geq \mu_i(\bar{x}_{-i})\left(f_i(x_i,\bar{x}_{-i})-\min_{y_i\in \mathbb{R}^n}f_i(y_i,\bar{x}_{-i})\right).
        \end{equation}
    %Additionally, assume that $\det\left(\nabla^2_{i i}f_i(x_i,\bar{x}_{-i})\right)\neq 0$ for all $x_i\in \mathbb{R}^{n}$ and for all $i$, \sa{Moreover, the game mapping $F(x)$ is norm coercive}. 
    Then, $\bar{x}$ is a Nash equilibrium of the game.
\end{theorem} 
\begin{proof}
By the definition of the quasi-Nash equilibrium, for $\bar x$ we have that 
\[\langle \nabla_i f(\bar x_i,\bar x_{-i}),x_i-\bar x_i\rangle\ge0 \qquad\hbox{for all $x_i\in K_i$ and for all $i$}.\]
Since $K_i=\mathbb{R}^n$ for all $i$, 
    it further follows that $\nabla_{i}f_i(\bar x_i,\bar{x}_{-i})=0$ for all $i$.
    Then, by the assumed condition in~\eqref{eq-plcond}, it follows that
    \[f_i(\bar x_i,\bar{x}_{-i})=\min_{y_i\in \mathbb{R}^n}f_i(y_i,\bar{x}_{-i})\qquad\hbox{for all $i$}.\]
    Therefore, 
    \[f_i(x_i,\bar x_{-i})\ge f_i(\bar x_i,\bar{x}_{-i})\qquad\hbox{for all $x_i\in\mathbb{R}^n$ and for all $i$},\]
    implying that $\bar x$ is a Nash equilibrium.
    \hfill$\square$
\end{proof}

The condition in~\eqref{eq-plcond} is motivated by the (P\L{}) condition, originally proposed and studied in \cite{POLYAK1963864} for nonconvex optimization problems. 
% This condition expresses a certain type of growth for a function \sa{\cite{karimi2016linear}}. 

In the following example, 
we provide a game where one can apply Theorem~\ref{thm-Normal-Mapping_Sol_Existence0} to assert the existence of a quasi-Nash equilibrium, while the existence results relying on the uniform P-matrix or P$_\Upsilon$-matrix condition cannot be applied (such as those of Proposition~2b), Proposition~3b), or Theorem~2 in~\cite{parise2019variational}). The game is constructed so that the game mapping is the same as the mapping $F(\cdot)$ considered in Example~\ref{example-vi}.

{\color{black}
\begin{example}\label{example-game}
    Consider a two player game with the following cost functions 
    \[f_1(x_1,x_2)=\frac{x_1^2}{2}+2x_1x_2,\quad f_2(x_1,x_2)=3x_1x_2+\frac{x_2^2}{2}\qquad\hbox{for all $x_1,x_2\in [0,\infty)$}.\]
   The game mapping is $F(x)=[x_1+2x_2,3x_1+x_2]^\top$, and the unique quasi-Nash equilibrium is $x^*=[0,0]^\top$.
    The Jacobian of the game mapping is \begin{equation*}\label{Example-Jacobian-game}
    \nabla F(x)=A,\qquad A=\begin{bmatrix}
1 &2 \\
3 &1 \\
\end{bmatrix} \qquad\hbox{for all } x.
\end{equation*}
As seen in Example~\ref{example-vi},
the VI associated with the game satisfies the conditions of Theorem~\ref{thm-Normal-Mapping_Sol_Existence0} and, thus, the game has a quasi-Nash equilibrium.
Note that the cost function of each player is strongly convex in the player's own decision variable,
so a quasi-Nash equilibrium is a Nash equilibrium. Hence, the game has a unique Nash equilibrium at $[0,0]^\top$.
\\
\noindent
{\it Uniform P-matrix condition fails.}
By Proposition~3c) in~\cite{parise2019variational}, if $\nabla F(\cdot)$ satisfies the uniform P-matrix condition, then $F(\cdot)$ is a uniform P-function. In Example~\ref{example-vi}, we have shown that the mapping $F(x)=Ax$ is not a uniform P-function, hence its Jacobian can not satisfy the uniform P-matrix condition.\\
{\it P$_\Upsilon$-matrix condition fails.}
In this case, with $2$-players and one-dimensional players' action sets, the matrix $A$ decomposes into four 1-dimensional blocks, so $\Upsilon=A$.
%in the $\rm P_\Upsilon$-matrix condition. 
The $\rm P_\Upsilon$-matrix condition requires that $A$ is a P-matrix, i.e., all principal minors are positive, which is not the case
since $\det(A)=-5$. \hfill $\square$
\end{example}}

\section{Conclusions}\label{Sec-Conclusion}
%---------------------------------------------------------
{\color{black} This paper has investigated sufficient conditions for the existence of solutions to a non-monotone VI$(K,F)$ by using the normal mapping $F_K^{\rm nor}(\cdot)$ associated with the VI$(K,F)$. In particular, the main results show that the solution set of the VI$(K,F)$ is nonempty and compact when the normal mapping $F_K^{\rm nor}(\cdot)$ is norm coercive and its generalized Jacobian has some 
non-singularity properties outside the set of zeros of the normal mapping $F_K^{\rm nor}(\cdot)$. Further, some conditions on the VI mapping and its generalized Jacobian are investigated ensuring that the assumptions of our main existence results are met. The implications of the main results for VIs arising from games have also been provided. 
}

\bibliographystyle{plain}
\bibliography{refs.bib}

@book{bno2003convex,
  title={Convex Analysis and Optimization},
  author={Bertsekas, Dimitri and Nedi\'c, Angelia and Ozdaglar, Asuman E},
  year={2003},
  publisher={Athena Scientific}
}

@book{konov,
author={Konnov, Igor},
title={Equilibrium Models and Variational Inequalities}, 
series={ser.
Mathematics in Science and Engineering, C. K. Chui ed.},
publisher={Amsterdam, The Netherlands: Elsevier}, year={2007}, 
volume={210}
}

@book{facchinei2003finite,
  title={Finite-dimensional Variational Inequalities and Complementarity Problems},
  author={Facchinei, Francisco and Pang, Jong-Shi},
  year={2003},
  publisher={Springer}
}

@article{C2,
  title={Beyond monotone variational inequalities: solution methods and iteration complexities},
  author={Huang, Kevin and Zhang, Shuzhong},
  journal={Pacific Journal of Optimization},
  volume={20},
  number={3},
  pages={403--428},
  year={2024},
  publisher={YOKOHAMA PUBL 101, 6-27 SATSUKIGAOKA AOBA-KU, YOKOHAMA, 227-0053, JAPAN}
}

@article{clarke1976inverse,
  title={On the inverse function theorem},
  author={Clarke, Francis H},
  journal={Pacific Journal of Mathematics},
  volume={64},
  number={1},
  pages={97--102},
  year={1976}
}

@book{KSbook,
  title={An Introduction to Variational Inequalities and Their Applications},
  author={Kinderlehrer, David and Stampacchia, Guido},
  year={2000},
  publisher={SIAM}
}

@article{parise2019variational,
  title={A variational inequality framework for network games: Existence, uniqueness, convergence and sensitivity analysis},
  author={Parise, Francesca and Ozdaglar, Asuman},
  journal={Games and Economic Behavior},
  volume={114},
  pages={47--82},
  year={2019},
  publisher={Elsevier}
}

@book{1,
  title={Topics in matrix analysis},
  author={Horn, Roger A and Johnson, Charles R},
  year={1994},
  publisher={Cambridge university press}
}

@article{pang2011nonconvex,
  title={Nonconvex games with side constraints},
  author={Pang, Jong-Shi and Scutari, Gesualdo},
  journal={SIAM Journal on Optimization},
  volume={21},
  number={4},
  pages={1491--1522},
  year={2011},
  publisher={SIAM}
}

@article{li2023stackelberg,
  title={Stackelberg and {N}ash equilibrium computation in non-convex leader-follower network aggregative games},
  author={Li, Rongjiang and Chen, Guanpu and Gan, Die and Gu, Haibo and L{\"u}, Jinhu},
  journal={IEEE Transactions on Circuits and Systems I: Regular Papers},
  volume={71},
  number={2},
  pages={898--909},
  year={2023},
  publisher={IEEE}
}

@article{carnevale2024tracking,
  title={Tracking-based distributed equilibrium seeking for aggregative games},
  author={Carnevale, Guido and Fabiani, Filippo and Fele, Filiberto and Margellos, Kostas and Notarstefano, Giuseppe},
  journal={IEEE Transactions on Automatic Control},
  volume={69},
  number={9},
  pages={6026--6041},
  year={2024},
  publisher={IEEE}
}

@inproceedings{arefizadeh2022distributed,
  title={A distributed algorithm for aggregative games on directed communication graphs},
  author={Arefizadeh, Sina and Nedi{\'c}, Angelia},
  booktitle={2022 IEEE 61st Conference on Decision and Control (CDC)},
  pages={6407--6412},
  year={2022},
  organization={IEEE}
}

@article{gohary2009generalized,
  title={A generalized iterative water-filling algorithm for distributed power control in the presence of a jammer},
  author={Gohary, Ramy H and Huang, Yao and Luo, Zhi-Quan and Pang, Jong-Shi},
  journal={IEEE Transactions on Signal Processing},
  volume={57},
  number={7},
  pages={2660--2674},
  year={2009},
  publisher={IEEE}
}

@article{hobbs2007nash,
  title={Nash-Cournot equilibria in electric power markets with piecewise linear demand functions and joint constraints},
  author={Hobbs, Benjamin F and Pang, Jong-Shi},
  journal={Operations Research},
  volume={55},
  number={1},
  pages={113--127},
  year={2007},
  publisher={INFORMS}
}

@book{luo1996mathematical,
  title={Mathematical Programs with Equilibrium Constraints},
  author={Luo, Zhi-Quan and Pang, Jong-Shi and Ralph, Daniel},
  year={1996},
  publisher={Cambridge University Press}
}

@article{scutari2010convex,
  title={Convex optimization, game theory, and variational inequality theory},
  author={Scutari, Gesualdo and Palomar, Daniel P and Facchinei, Francisco and Pang, Jong-Shi},
  journal={IEEE Signal Processing Magazine},
  volume={27},
  number={3},
  pages={35--49},
  year={2010},
  publisher={IEEE}
}

@unpublished{melo2018variational,
  title={A variational approach to network games},
  author={Melo, Emerson},
  note={Working Paper, No. 005.2018},
  year={2018}
}

@inproceedings{naghizadeh2017uniqueness,
  title={On the uniqueness and stability of equilibria of network games},
  author={Naghizadeh, Parinaz and Liu, Mingyan},
  booktitle={2017 55th Annual Allerton Conference on Communication, Control, and Computing (Allerton)},
  pages={280--286},
  year={2017},
  organization={IEEE}
}

@inproceedings{NonMonotoneVI,
  title={Non-Monotone Variational Inequalities},
  author={Arefizadeh, Sina and Nedi{\'c}, Angelia},
  booktitle={2024 60th Annual Allerton Conference on Communication, Control, and Computing (Allerton)},
  pages={1--7},
  year={2024},
  organization={IEEE}
}

@article{ebrahimi2025united,
  title={United We fall: On the {N}ash equilibria of multiplex and multilayer network games},
  author={Ebrahimi, Raman and Naghizadeh, Parinaz},
  journal={IEEE Transactions on Control of Network Systems},
  pages={1238--1250},
  year={2025},
  publisher={IEEE}
}

@inproceedings{ratliff2013characterization,
  title={Characterization and computation of local {N}ash equilibria in continuous games},
  author={Ratliff, Lillian J and Burden, Samuel A and Sastry, Shankar S},
  booktitle={2013 51st Annual Allerton Conference on Communication, Control, and Computing (Allerton)},
  pages={917--924},
  year={2013},
  organization={IEEE}
}

@article{rosen1965existence,
  title={Existence and uniqueness of equilibrium points for concave n-person games},
  author={Rosen, Judah Ben},
  journal={Econometrica: Journal of the Econometric Society},
  pages={520--534},
  year={1965},
  publisher={JSTOR}
}

@book{clarke1990optimization,
  title={Optimization and Nonsmooth Analysis},
  author={Clarke, Frank H},
  year={1990},
  publisher={SIAM}
}

@article{imbert2002support,
  title={Support functions of the {C}larke generalized {J}acobian and of its plenary hull},
  author={Imbert, Cyril},
  journal={Nonlinear Analysis: Theory, Methods \& Applications},
  volume={49},
  number={8},
  pages={1111--1125},
  year={2002},
  publisher={Elsevier}
}

@book{meyer2023matrix,
  title={Matrix Analysis and Applied Linear Algebra},
  author={Meyer, Carl D},
  year={2023},
  publisher={SIAM}
}

@article{ravat2011characterization,
  title={On the characterization of solution sets of smooth and nonsmooth convex stochastic {N}ash games},
  author={Ravat, Uma and Shanbhag, Uday V},
  journal={SIAM Journal on Optimization},
  volume={21},
  number={3},
  pages={1168--1199},
  year={2011},
  publisher={SIAM}
}

@article{simon1955behavioral,
  title={A behavioral model of rational choice},
  author={Simon, Herbert A},
  journal={The quarterly journal of economics},
  pages={99--118},
  year={1955},
  publisher={JSTOR}
}

@article{pang2013joint,
  title={Joint sensing and power allocation in nonconvex cognitive radio games: {Q}uasi-{N}ash equilibria},
  author={Pang, Jong-Shi and Scutari, Gesualdo},
  journal={IEEE Transactions on Signal Processing},
  volume={61},
  number={9},
  pages={2366--2382},
  year={2013},
  publisher={IEEE}
}

@unpublished{2510.02724,
Author = {Sina Arefizadeh and Angelia Nedić},
Title = {On Non-Monotone Variational Inequalities},
Year = {2025},
note = {Preprint arXiv:2510.02724},
}

@book{aleskerov2007utility,
  title={Utility maximization, choice and preference},
  author={Aleskerov, Fuad and Bouyssou, Denis and Monjardet, Bernard},
  year={2007},
  publisher={Springer}
}

@article{pales2007infinite,
  title={Infinite dimensional {C}larke generalized {J}acobian},
  author={P{\'a}les, Zsolt and Zeidan, Vera},
  journal={Journal of Convex Analysis},
  volume={14},
  number={2},
  pages={433--454},
  year={2007},
  publisher={HELDERMANN VERLAG LANGER GRABEN 17, 32657 LEMGO, GERMANY}
}

@article{pales2008infinite,
  title={Infinite dimensional generalized Jacobian: properties and calculus rules},
  author={P{\'a}les, Zsolt and Zeidan, Vera},
  journal={Journal of mathematical analysis and applications},
  volume={344},
  number={1},
  pages={55--75},
  year={2008},
  publisher={Elsevier}
}

@article{fiedler-ptak-Pmatrices,
author={Miroslav Fiedler and Vlastimil Ptak}, 
title={On matrices with non-positive off-diagonal elements and positive principal minors}, 
journal={Czechoslovak Mathematical Journal},
volume={12}, 
pages={382--400},
year={1962}
}

@article{POLYAK1963864,
title = {Gradient methods for the minimisation of functionals},
journal = {USSR Computational Mathematics and Mathematical Physics},
volume = {3},
number = {4},
pages = {864-878},
year = {1963},
issn = {0041-5553},
author = {Polyak, Boris T}
}

@inproceedings{qu2025generalized,
  title={Generalized {N}ash Equilibrium Seeking for Aggregative Games Over Weight-Unbalanced Digraphs},
  author={Qu, Qiangren and Xiao, Yao and Gang, Tieqiang and Chen, Lijie},
  booktitle={2025 37th Chinese Control and Decision Conference (CCDC)},
  pages={2593--2599},
  year={2025},
  organization={IEEE}
}

@article{ran2024distributed,
  title={Distributed generalized {N}ash equilibria computation of noncooperative games via novel primal-dual splitting algorithms},
  author={Ran, Liang and Li, Huaqing and Zheng, Lifeng and Li, Jun and Li, Zhe and Hu, Jinhui},
  journal={IEEE Transactions on Signal and Information Processing over Networks},
  volume={10},
  pages={179--194},
  year={2024},
  publisher={IEEE}
}
\end{document}